\documentclass{article}
\usepackage{cite}
\usepackage{amsmath,amssymb,amsthm,mathrsfs}
\usepackage{titlesec,hyperref}
\usepackage{color}
\usepackage{fancyhdr}
\usepackage{cases}
\usepackage[switch]{lineno}
\usepackage[margin=2.5cm]{geometry}
\usepackage{enumerate}
\usepackage[integrals]{wasysym}
\usepackage{bm}
\usepackage{graphicx}
\usepackage{caption}
\usepackage{float}
\usepackage{subfigure}
\usepackage{titlesec}
\usepackage{multirow}
\usepackage{booktabs}
\newcommand{\dd}{\mathop{}\!\mathrm{d}}
\let\del\partial

\newcommand{\xaddspace}[0]{\mathchoice{\hspace{-0.8em}}
                {\hspace{-0.6em}}
                {\hspace{-0.4em}}
                {\hspace{-0.3em}}
                 }

\newcommand{\xint}[1]{\int\foreach \i in {2,...,#1}{\xaddspace\int}}
\newcommand{\myiint}[0]{\xint{2}}
\renewcommand{\iint}{\myiint}

\newcommand{\TTT}[0]{\mathbb{T}}
\newcommand{\R}[0]{\mathbb{R}}

\newcommand{\MM}{\mathring M}
\newcommand{\MMM}{\mathring{\bar M}}
\newcommand{\RR}{\mathring R}
\newcommand{\RRR}{\mathring{\bar R}}
\newcommand{\ootimes}{\mathbin{\mathring{\otimes}}}

\let\div\relax
\DeclareMathOperator{\div}{div}
\DeclareMathOperator{\tr}{Tr}
\newcommand{\vin}[0]{v^{\textup{in}}}
\newcommand{\bin}[0]{b^{\textup{in}}}
\newcommand{\pex}[0]{\textsf{\textit{p}}}
\newcommand{\vex}[0]{\textsf{\textit{v}}}
\newcommand{\bex}[0]{\textsf{\textit{b}}}
\newcommand{\vv}[0]{\bar v}
\newcommand{\bb}[0]{\bar b}

\newcommand{\ppp}[0]{\bar{\bar p}}
\newcommand{\pp}[0]{\bar p}
\newcommand{\wpqu}[1][q+1]{w^{\textup{(pu)}}_{#1}}
\newcommand{\wpqb}[1][q+1]{w^{\textup{(pb)}}_{#1}}
\newcommand{\wpq}[1][q+1]{w^{\textup{(p)}}_{#1}}
\newcommand{\wcq}[1][q+1]{w^{\textup{(c)}}_{#1}}
\newcommand{\wlq}[1][q+1]{w^{\Delta}_{#1}}

\newcommand{\wsq}[1][q+1]{w^{\textup{(s)}}_{#1}}
\newcommand{\wtq}[1][q+1]{w^{\textup{(t)}}_{#1}}
\newcommand{\wvq}[1][q+1]{w^{\textup{(v)}}_{#1}}
\newcommand{\whq}[1][q+1]{w^{\textup{(h)}}_{#1}}
\newcommand{\dvq}[1][q+1]{d^{\textup{(v)}}_{#1}}
\newcommand{\dhq}[1][q+1]{d^{\textup{(h)}}_{#1}}
\newcommand{\dsq}[1][q+1]{d^{\textup{(s)}}_{#1}}
\newcommand{\dlq}[1][q+1]{d^{\Delta}_{#1}}

\newcommand{\dtq}[1][q+1]{d^{\textup{(t)}}_{#1}}
\newcommand{\dpq}[1][q+1]{d^{\textup{(p)}}_{#1}}
\newcommand{\dcq}[1][q+1]{d^{\textup{(c)}}_{#1}}

\newcommand{\Rosc}{R_{q+1}^\textup{osc}}
\newcommand{\Rlinear}{R_{q+1}^\textup{lin}}

\newcommand{\Mosc}{M_{q+1}^\textup{osc}}
\newcommand{\Mlinear}{M_{q+1}^\textup{lin}}
\newcommand{\TT}[0]{\mathsf{T}}
\def\dashint{\,\ThisStyle{\ensurestackMath{%
  \stackinset{c}{.2\LMpt}{c}{.5\LMpt}{\SavedStyle-}{\SavedStyle\phantom{\int}}}%
  \setbox0=\hbox{$\SavedStyle\int\,$}\kern-\wd0}\int}
\def\ddashint{\,\ThisStyle{\ensurestackMath{%
  \stackinset{c}{.2\LMpt}{c}{.5\LMpt+.2\LMex}{\SavedStyle-}{%
    \stackinset{c}{.2\LMpt}{c}{.5\LMpt-.2\LMex}{\SavedStyle-}{%
      \SavedStyle\phantom{\int}}}}\setbox0=\hbox{$\SavedStyle\int\,$}\kern-\wd0}\int}

\newcommand{\coloneq}{\mathrel{\mathop:}=}

\DeclareMathOperator{\curl}{curl}

\newtheorem{thm}{Theorem}[section]
\newtheorem{cor}[thm]{Corollary}

\newtheorem{lem}[thm]{Lemma}
\newtheorem{prop}[thm]{Proposition}
\theoremstyle{definition}
\newtheorem{defn}[thm]{Definition}

\theoremstyle{remark}
\newtheorem{rem}[thm]{Remark}
\numberwithin{equation}{section}

\allowdisplaybreaks

\begin{document}
\title{On the weak solutions for the MHD systems with controllable total energy and cross helicity}

\author{
Changxing $\mbox{Miao}^{1}$ \footnote{email: miao changxing@iapcm.ac.cn}
\quad and\quad
Weikui $\mbox{Ye}^{2}$ \footnote{email:
904817751@qq.com}\\
$^1\mbox{Institute}$ of Applied Physics and Computational Mathematics,\\
P.O. Box 8009, Beijing 100088, P. R. China\\
$^2\mbox{School}$ of Mathematical Sciences, South China Normal University,\\
Guangzhou, 510631, China\\
}
\date{}
\maketitle
\begin{abstract}
In this paper, we prove the non-uniqueness of three-dimensional  magneto-hydrodynamic (MHD) system in $C([0,T];L^2(\mathbb{T}^3))$ for any initial data in  $H^{\bar{\beta}}(\mathbb{T}^3)$~($\bar{\beta}>0$), by exhibiting that the total energy and the cross helicity can be controlled in a given positive time interval. Our results extend the non-uniqueness results of the ideal MHD system  to the viscous and resistive MHD system. Different from the ideal MHD system, the dissipative effect in the viscous and resistive MHD system prevents the nonlinear term from balancing the stress error $(\RR_q,\MM_q)$ as doing in \cite{2Beekie}. We introduce the box flows and construct the perturbation consisting in seven different kinds of flows in convex integral scheme,  which ensures that the iteration works and yields the non-uniqueness.
\end{abstract}
\noindent \textit{Keywords}: convex integral iteration, the MHD system, weak solutions, cross helicity, non-uniqueness.\\
Mathematics Subject Classification: 35A02,~35D30,~35Q30,~76D05,~76W05.
\tableofcontents

\section{Introduction}
\label{s:intro}
In this paper, we consider the Cauchy problem of the following 3D MHD equations:
\begin{equation}
\left\{ \begin{alignedat}{-1}
\del_t v-\nu_1\Delta v+ \div (v\otimes v)  +\nabla p   &=\div (b\otimes b) ,
 \\
 \del_t b-\nu_2\Delta b+\div (v\otimes b)     &= \div (b\otimes v) ,
 \\
  \div v = 0,~\div b&= 0,
  \\ (v,b) |_{t=0}&=(\vin,\bin),
\end{alignedat}\right.  \label{mhd}
\end{equation}
where $v(t,x)$ is the fluid velocity, $b(t,x)$ is the magnetic fields and $p(t,x)$ is the scalar pressure. $\nu_1$ and $\nu_2$ are the viscous and resistive coefficients, respectively. We call \eqref{mhd} the viscous and resistive MHD system when $\nu_1,\nu_2>0$. For a given initial data $(\vin,\bin)\in H^{\bar{\beta}}(\TTT^3)$ \footnote{Here and throughout the paper, we denote $\mathbb{T}^3=[0,1]^3$ and $(f , g)\in X\times X$ by $(f, g)\in X$} with $\bar{\beta}>0$, we construct a weak solution of \eqref{mhd} with controllable total energy and cross helicity, which implies the non-uniqueness of weak solutions in $C([0,T];L^2(\mathbb{T}^3))$.

To begin with,  let us introduce the definition of weak solutions of \eqref{mhd}.
\begin{defn} Let $(\vin,\bin)\in L^2(\TTT^3) $. We say that  $(v,b)\in C([0,T];L^2(\mathbb{T}^3)) $ is a weak solution to \eqref{mhd}, if $\div v=\div b=0$ in the weak sense, and  for all divergence-free test functions $\phi\in C^\infty_0([0,T)\times \mathbb T^3)$,
\begin{align}
\int_0^T \int_{\mathbb T^3} (\del_t-\nu_1\Delta)\phi v+\nabla \phi : (v\otimes v-b\otimes b) \dd x \dd t=  -\int_{\mathbb T^3} \vin \phi(0,x)\dd x,\\
\int_0^T \int_{\mathbb T^3} (\del_t-\nu_2\Delta)\phi b+\nabla \phi : (v\otimes b-b\otimes v) \dd x \dd t=-\int_{\mathbb T^3} \bin \phi(0,x)\dd x.
\end{align}
\end{defn}

When $\nu_1=\nu_2=0$, \eqref{mhd} becomes the ideal MHD system:
\begin{equation}
\left\{ \begin{alignedat}{-1}
\del_t v+ \div (v\otimes v)  +\nabla p   &=\div (b\otimes b) ,
 \\
 \del_t b+\div (v\otimes b)     &= \div (b\otimes v) ,
 \\
  \div v = 0,~\div b&= 0,
  \\ (v,b) |_{t=0}&=(\vin,\bin).
\end{alignedat}\right.  \label{idealmhd}
\end{equation}
For the smooth solutions to \eqref{idealmhd}, they possess a number of physical invariants :

The total  energy: $e(t) = \int_{\mathbb{T}^3} (|v(t,x)|^2 + |b(t,x)|^2)\dd x;$

The cross helicity: $h_{v,b}(t) = \int_{\mathbb{T}^3} (v(t,x)\cdot b(t,x))\dd x;$

The magnetic helicity: $h_{b,b}(t) = \int_{\mathbb{T}^3} (A(t,x)\cdot b(t,x))\dd x,$\\
where $A$ is a periodic vector field with  zero mean  satisfying $\curl A = b$.

{In 1949, Lars Onsager conjectured that the H\"{o}lder exponent threshold for the energy conservation of weak solutions of the Euler equations is  $1/3$. Since then, many mathematicians are devoted to proving Onsager conjecture on the Euler equations and there have been a flood of papers with  this problem \cite{cwe,MV,NV,27DeLellis,28DeLellis,Isett,zbMATH06312794,zbMATH06710292,zbMATH07370998,Rosa2021DimensionOT}.
In recent years, Onsager-type conjectures on the ideal MHD equations which possess several physical invariants  have caught researchers' interest and some progress has been made on related issues.} For instance, in \cite{1Hydrodynamic,34Caflisch,45Kang}, the magnetic helicity conservation for the 3D ideal MHD was proved in the critical space $L^3_{t,x}$. Later, Faraco-Lindberg-Sz\'{e}kelyhidi
\cite{37Faraco} showed that the $L^3_{t,x}$
integrability condition for the magnetic helicity conservation is sharp.
The cross helicity and total energy are conservative when the weak solutions $(v,b)\in L^3_tB^{\alpha}_{3,\infty}$ with $\alpha > \frac{1}{3}$, but whether they are conservative for $\alpha\leq \frac{1}{3}$ or not is still an open problem. Throughout current literatures, whether the solution satisfies the physical invariants or not plays a key role in studying the non-uniqueness problems. 
In \cite{36Faraco}, Faraco-Lindberg-Sz\'{e}kelyhidi constructed non-trivial weak solutions with compact support in space $L^{\infty}_{t,x}$. 
Beekie-Buckmaster-Vicol \cite{2Beekie} constructed distributional solutions in $C_tL^2_x$ breaking the magnetic helicity conservation, and showed the non-uniqueness of weak solutions. 

System \eqref{idealmhd} with $b\equiv0$ becomes the famous Euler equations. Many authors are devoted to the study of the non-uniqueness issue on Euler equations.
In the pioneering paper \cite{27DeLellis}, De Lellis-Sz\'{e}kelyhidi developed the convex integration scheme
and constructed the weak solutions in $L^{\infty}_{t,x}$ with compact support to the 3D Euler equations, see
also \cite{28DeLellis}. After that, there have been a series of results on non-uniqueness of weak solutions.  The Onsager conjecture was finally solved in $
C^{\beta}_{x,t}~ (0 < \beta < 1/3)$ by Isett \cite{Isett}, and by
Buckmaster-De Lellis-Sz\'{e}kelyhidi-Vicol \cite{zbMATH07038033} for admissible weak solutions. {Very recently, Daneri \cite{zbMATH06312794}, subsequently in \cite{zbMATH06710292} and  Runa \cite{zbMATH07370998}  considered non-uniqueness problem by constructing wild initial data which is $L^2(\mathbb T^3)$-dense, while Rosa and Haffter \cite{Rosa2021DimensionOT} also showed that any smooth initial data gives rise to uncountably many solutions.}

For the incompressible Navier-Stokes equations \eqref{mhd} with $b\equiv0$, there have many results on the non-uniqueness problems. Buckmaster-Vicol in \cite{13Nonuniqueness} made the first important break-through by making use of a
${L^2_x}$
-based intermittent convex integration scheme. Subsequently, Buckmaster, Colombo and  Vicol \cite{2018Wild} showed that the wild solutions can be generated by  $H^3$ initial data.
 Recently, another non-uniqueness
result based on Serrin condition for the Navier-Stokes equations was proved by Cheskidov-Luo \cite{1Cheskidov}, which shows the sharpness of the Ladyzhenskaya-Prodi-Serrin criteria $\frac{2}{p}+\frac{d}{q}\le 1$ at the endpoint $(p,q)=(2,\infty)$. In \cite{leray},  Albritton, Bru\'{e} and Colombo proved the non-uniqueness of the Leray-Hopf solutions with a special force by skillfully constructing a ``background'' solution which is unstable for the Navier-Stokes dynamics in similarity variables. For the 3D hyper-viscous NSE, Luo-Titi \cite{luotianwen} also proved the non-uniqueness results, whenever the exponent of viscosity is less than the Lions exponent 5/4.

For the viscous and resistive MHD system, the existence of Leray-Hopf solutions to the MHD equations was proved by Wu \cite{67Wu}. In \cite{lyc}, Li, Zeng and Zhang proved the non-uniqueness of weak solutions in $H^{\epsilon}_{t,x}$, where $\epsilon$ sufficiently small. However, the uniqueness of Leray-Hopf solutions is unsolved, even in $C_tL^2_x\cap L^2_t\dot{H}^1_x$ is still open. In \cite{13Nonuniqueness}, the non-uniqueness result for the Navier-Stokes equations also imply the non-uniqueness of the viscous and resistive MHD system with trivial magnetic  field in $C_tL^2_x$. One natural problem is  {\emph{whether the viscous and resistive MHD system with non-trivial magnetic fields in $C_tL^2_x$ is unique or not}}. In this paper, we solve this problem by showing the non-uniqueness of \eqref{mhd} with $\nu_1,\nu_2>0$. Now we are in position to state the main result.
\begin{thm}\label{t:main0}
A weak solution $(v,b)$ of the viscous and resistive MHD system in $C([0,T];L^2(\mathbb{T}^3))$ is non-unique if $(v,b)$ has at least one interval of regularity. Moreover, there exist non-Leray-Hopf weak solutions $(v,b)$ in $C([0,T];L^2(\mathbb{T}^3) )$.
\end{thm}
\begin{rem}
{For the ideal MHD system \eqref{idealmhd} with non-trivial magnetic fields, Beekie-Buckmaster-Vicol in \cite{2Beekie} proved the non-uniqueness for the weak solutions in $C_tL^2_x$. However, for the viscous and resistive MHD system \eqref{mhd}, the uniqueness for solutions in $C_tL^2_x$ is still unsloved.  Theorem \ref{t:main0} solves this problem and extends the non-uniqueness results of the ideal MHD system to the viscous and resistive MHD system.}

  Compared with the ideal MHD system, the dissipative effect prevents the nonlinear term from balancing the stress error $(\RR_q,\MM_q)$ as doing in \cite{2Beekie}. This
 leads to the major difficulty in convex integral iteration in $C_tL^2_x$. A nature choice is using 3D box type flows instead of the Mikado flows in convex integral iteration. However, these 3D box type flows do not have enough freedom on the oscillation directions in the velocity and magnetic flows, which will give rise to additional errors in the oscillation terms. Inspired by \cite{1Cheskidov,2Beekie,13Nonuniqueness}, we construct ``temporal flows'' and ``Inverse traveling wave flows'' to eliminate these extra errors, which help us construct a weak solution by combining with the principal flows. Moreover, we construct the so-called ``Initial flows'' and ``Helicity flows'' to achieve
$$(v(0,x),b(0,x))=(\vin,\bin),~~\text{and}~~\int_{\mathbb T^3}(|v|^2+|b|^2)\dd x=e(t),~~\text{and}~~~\int_{\mathbb T^3}v\cdot b\dd x=h(t),~~t\in[1,T],$$
which yields the non-uniqueness the weak solution.
\end{rem}

We now present a main theorem, which immediately implies  Theorem \ref{t:main0} by showing that the total energy and the cross helicity can be controlled in a given positive time interval:
\begin{thm}[Main theorem]\label{t:main}
Let $T,\bar{\beta}>0$ and $(\vin,\bin) \in H^{\bar{\beta}} (\mathbb T^3)$. For fixed $\delta_2>0$, assume that there exists two smooth functions $e(t),h(t)$ satisfying
\begin{align}\label{helicity-0}
\frac{\delta_{2} }{2} &\le e(t) -\int_{\mathbb T^3} (|\vin|^2+|\bin|^2) \dd x \le \frac{3\delta_{2} }{4},~~t\in[\tfrac{1}{2},T]
\end{align}
and
\begin{align}\label{helicity-1}
\frac{\delta_{2} }{200} &\le h(t) -\int_{\mathbb T^3} \vin\cdot\bin \dd x \le \frac{\delta_{2} }{50},~~t\in[\tfrac{1}{2},T].
\end{align}
Then there exists a weak solution $(v,b)\in C([0,T];L^2(\mathbb T^3))$ to the viscous and resistive MHD system with initial data $(\vin,\bin)$. Moreover,  we have
$$\int_{\mathbb T^3}(|v|^2+|b|^2)\dd x=e(t)~\text{and}~\int_{\mathbb T^3}v\cdot b\dd x=h(t),~~t\in[1,T]$$
where $h(t):=h_{v,b}(t)$ denotes the cross helicity.
\end{thm}
\begin{rem}
For a given $(\vin,\bin) \in H^{\bar{\beta}} $, one can choose infinitely many functions $e(t),h(t)$ satisfying \eqref{helicity-0} and \eqref{helicity-1}, which implies the non-uniqueness of weak solutions. Moreover, we will prove that $(v,b)\in  C([0,T];H^{\epsilon}(\mathbb T^3))$ with $0<\epsilon\ll\bar{\beta}$ in Section 2.2.
\end{rem}

For the ideal MHD system \eqref{idealmhd}, one can obtain a similar result after a simple modification to the proof of Theorem~\ref{t:main}.
\begin{thm}\label{ideal mhd}
Let $T,\bar{\beta}>0$ and $(\vin,\bin) \in H^{\bar{\beta}} (\mathbb T^3)$. Then there exist infinitely many smooth functions $e(t),h(t)$ associated with a weak solution $(v,b)\in C([0,T];L^2(\mathbb T^3))$ to the ideal MHD system \eqref{idealmhd} with initial data $(\vin,\bin)$. Moreover, we have
$$\int_{\mathbb T^3}(|v|^2+|b|^2)\dd x=e(t)~~\text{and}~~\int_{\mathbb T^3}v\cdot b\dd x=h(t),~~~t\in[1,T].$$
\end{thm}
\begin{rem}
 Theorem \ref{ideal mhd} shows that all initial data in $H^{\bar{\beta}}~(\forall\bar{\beta}>0)$ may generate non unique  weak solutions by choosing different total energy $e(t)$ or cross helicity $h(t)$. {For weak solutions with non-conservative magnetic helicity, one can see \cite{37Faraco,2Beekie,lyc} for more details.} 
\end{rem}

As a matter of fact, authors in \cite{2Beekie} constructed solutions in $C_tH^{\epsilon}_x\hookrightarrow C_tL^2_x$ which breaks the conservative law of magnetic helicity.  In view of Taylor's conjecture, these weak solutions cannot be the weak ideal limits of Leray-Hopf  weak solutions. Mathematically, Taylor's conjecture is stated as follows:
\begin{thm}[Taylor's conjecture\cite{35Faraco,65Taylor}]\label{Taylor}
Suppose that $(v,b)\in L^{\infty}([0,T];L^2(\mathbb T^3))$ is a weak ideal limit of sequence of Leray-Holf weak solutions of the viscous and resistive MHD system, then the magnetic helicity is conservative.
\end{thm}
Fortunately, combining Theorem \ref{t:main} with Theorem \ref{ideal mhd}, we can prove that the weak solutions constructed in \cite{2Beekie} can be a vanishing viscosity and resistivity limit of the weak solutions to \eqref{mhd}, which is similar to Theorem 1.3 in \cite{13Nonuniqueness}.
\begin{cor}\label{vanishing}
Suppose that $(v,b)\in C([0,T];H^{\epsilon}(\mathbb T^3))$ is a weak solution of \eqref{idealmhd}. Then, there exist $0<\epsilon'\ll\epsilon$  and a
sequence of weak solutions $(v^{\nu_n},b^{\nu_n})\in C([0,T];H^{\epsilon'}(\mathbb T^3))$ to  the viscous and resistive MHD system
such that,
$$(v^{\nu_n},b^{\nu_n})\rightarrow (v,b)~~strongly~~ in~~ C_tL^2_x,~ ~as~~ \nu_n\rightarrow 0,$$
where $\nu_n=(\nu_{1,n},\nu_{2,n})$.
\end{cor}

\section{Outline of the convex integration scheme}
\label{s:outline}

In this paper, it suffices to prove Theorem \ref{t:main} for \eqref{mhd} with $\nu_1,\nu_2>0$. Without loss of generality, we set $\nu_1=\nu_2=1$ .

\subsection{Parameters and the iterative process}
\label{ss:params}
Set $\bar{\beta}<1$. If $(\vin, \bin)$ is sufficiently smooth, we still have $(\vin, \bin)\in H^{1}\subset H^{\bar{\beta}}$.  We choose $b=2^{16\lceil\bar{\beta}^{-1/2}\rceil} $, $\beta=\frac{\bar{\beta}}{b^4}$, $\alpha$ to be a small constant depending on $b,\beta,\bar{\beta}$ such that $0<\alpha\leq\min\{\frac{1}{b^6},\frac{\beta}{b^3}\}$, and $a\in\mathbb{N}^+$ to be a large number depending on $b,\beta,\bar{\beta},\alpha$ and the initial data . 
We define
\begin{alignat}{10}
    \lambda_q &\coloneq  a^{(b^q)} ,&\qquad   \delta_q &\coloneq \lambda_2^{3\beta}\lambda_q^{-2\beta},~~q\in\mathbb{N}^+.
\end{alignat}
For $q=1,2$, $\delta_q$ is a large number which could bound the $L^2$ norm of initial data by choosing $a$ sufficiently large. For~$q\geq3$, $\delta_q$ is small and tends to zero as $q\rightarrow\infty$.

Firstly, we choose two smooth functions $e: [1/2,T]\to [0,\infty)$, $h:[1/2,T]\to (-\infty,\infty)$ such that
\begin{align}
\frac{\delta_{2} }{2} &\le e(t) -\int_{\mathbb T^3}( |\vin|^2+|\bin|^2) \dd x \le \frac{3\delta_{2} }{4} ,
    \label{e:vin-energy-estimate} \\
    \frac{\delta_{2} }{200} &\le h(t) -\int_{\mathbb T^3} \vin\cdot\bin \dd x \le \frac{\delta_{2} }{150}.
    \label{e:bin-helicity-estimate}
\end{align}

Secondly, adopting strategy of convex integration scheme, we consider a modification of \eqref{mhd} with stress tensor error $(\RR_q,\MM_q)$. Assume that $\psi_{\epsilon}:=\tfrac{1}{\epsilon}\widetilde{\psi}(\tfrac{x}{\epsilon})$ stands for a sequence of standard mollifiers, where $\widetilde{\psi}$ is a non-negative radial bump function. Let $(v_q,b_q,p_q,\RR_q,\MM_q)$ solve
\begin{equation}
\left\{ \begin{alignedat}{-1}
\del_t v_q-\Delta v_q+\div (v_q\otimes v_q)  +\nabla p_q   &=\div (b_q\otimes b_q) + \div \RR_q,
 \\
 \del_t b_q-\Delta b_q+\div (v_q\otimes b_q)     &= \div (b_q\otimes v_q)+ \div \MM_q,
 \\
  \nabla \cdot v_q &= 0,~\nabla \cdot b_q = 0,
  \\ (v_q,b_q) |_{t=0}&=(\vin* \psi_{\ell_{q-1}},\bin* \psi_{\ell_{q-1}}).
\end{alignedat}\right.  \label{e:subsol-euler}
\end{equation}
where $\ell_{q-1}:=\lambda^{-6}_{q-1}$,~ $v \otimes b\coloneq  (v_j b_i)_{i,j=1}^3$, and vector  $\div M$ denotes the divergence of a 2-tensor $M=(M_{ij})_{i,j=1}^3$  with components:
\begin{align*}
(\div M)_i \coloneq   \partial_j M_{ji}.
\end{align*}
In particular, $\div(v\otimes b)=(v\cdot\nabla )b$ if $\div v=0$. The magnetic stress $\MM_q$ is required to be an anti-symmetric matrix. And the Reynolds stress $\RR_q$ is a symmetric, trace-free $3\times3$ matrix.
\begin{gather}
     \MM_q = -\MM_q^\TT,~~~~ \RR_q = \RR_q^\TT ,  ~~~~ \tr\RR_q = \sum_{i=1}^3 (\RR_q)_{ii} =  0. \label{e:RR-cond}
\end{gather}

The estimates we propagate inductively  are:
\begin{align}
    \|(v_q,~b_q)\|_{L^2} &\le C_0\sum_{l=1}^q\delta^{1/2}_l ,
    \label{e:vq-C0}
    \\
    \|(v_q,~b_q)\|_{H^3} &\le  \lambda^5_q ,
    \label{e:vq-C1}
    \\
    \|(\RR_q,~\MM_q)\|_{L^1} &\le \delta_{q+1}\lambda_q^{-40\alpha} ,
    \label{e:RR_q-C0}
\\
(v_q(0,x),~b_q(0,x))&=(\vin * \psi_{\ell_{q-1}},~\bin* \psi_{\ell_{q-1}}),
    \label{e:RR_q-C1}
    \\
   t\in[ 1-\tau_{q-1},T] \implies \tfrac{1}{3}\delta_{q+1} &\le e(t) -\int_{\mathbb T^3} |v_q|^2+|b_q|^2 \dd x \le \delta_{q+1} ,
    \label{e:energy-q-estimate}
    \\
    t\in[ 1-\tau_{q-1},T] \implies \tfrac{\delta_{q+1}}{300} &\le h(t) -\int_{\mathbb T^3}  v_q\cdot b_q\dd x \le \tfrac{\delta_{q+1}}{100} ,
    \label{e:energy-h-estimate}
\end{align}
where ${\ell_q} \coloneq  \lambda_q^{-6}$, $\tau_q \coloneq  {\ell}^3_q$ and $C_0\coloneq 600$. By the definition of $\delta_q$, one can easily deduce that $\sum_{i=1}^{\infty}\delta_i$ converges to a finite number. Moreover, we restrict the error of the cross helicity to be much  smaller than the energy error {in the iterative procedure}, which is used to reduce the impact on the energy error, see Section 4.4.

\begin{prop}
\label{p:main-prop}Let $(\vin,\bin)\in H^{\bar{\beta}}(\mathbb T^3)$ with $0<\bar{\beta}<1$.
Assume that $(v_q,b_q,p_q,\MM_q,\RR_q)$ solves
\eqref{e:subsol-euler} and satisfies \eqref{e:vq-C0}--\eqref{e:RR_q-C1}, and $e(t),h(t)$ are any smooth functions satisfying \eqref{e:energy-q-estimate}-\eqref{e:energy-h-estimate},
then there exists a solution $(v_{q+1},b_{q+1}, p_{q+1}, \RR_{q+1}, \MM_{q+1})$, satisfying \eqref{e:subsol-euler},
\eqref{e:vq-C0}--\eqref{e:energy-h-estimate}
with $q$ replaced by $q+1$, and such that
\begin{align}
        \|(v_{q+1} - v_q,~b_{q+1} - b_q)\|_{L^2} &\leq C_0 \delta_{q+1}^{1/2}.
        \label{e:velocity-diff}
\end{align}
\end{prop}
\textbf{Notations:} Throughout this paper,  we set that
\begin{align*}
v\mathring\otimes b:=v\otimes b-\tfrac{1}{3}\tr (v \otimes b){\rm Id},~~~~\mathbb{P}_{H}:={\rm Id}-\frac{\nabla\div}{\Delta},~~~~{\mathbb{P}_{>0}}f(x)=f(x)-\int_{\mathbb{T}^3}f(z)\dd z,
\end{align*}
where $v\mathring\otimes b$ is a trace-free matrix and $\frac{f}{\Delta}(z):=\sum_{l\in\mathbb{Z}^3/\{0\}}\frac{\hat{f}_l}{-l^2} e^{i l\cdot z}$ for any mean free function $f$. We have $\frac{\partial_jf}{\Delta}(z)=\sum_{l\in\mathbb{Z}^3/\{0\}}\frac{il_j\hat{f}_l}{-l^2} e^{i l\cdot z}$.

Next, we prove that Proposition \ref{p:main-prop} implies Theorem~\ref{t:main}. To start the iteration, we define $(v_1,b_1,p_1,\RR_1,\MM_1)$ by	
\begin{gather*}
    v_1(x,t)\coloneq e^{t\Delta}\vin * \psi_{\ell_{0}},~~b_1(x,t)\coloneq e^{t\Delta}\bin(x)* \psi_{\ell_{0}},~~ p_1(x,t)\coloneq |v_1 |^2-|b_1|^2,\\
    \RR_1(x,t)\coloneq v_1 \mathring\otimes v_1 -b_1 \mathring\otimes b_1,~~
         \MM_1(x,t)\coloneq v_1 \otimes b_1-b_1 \otimes v_1.
\end{gather*}
It is easy to verify that $(v_1,b_1,p_1,\RR_1,\MM_1)$ solves \eqref{e:subsol-euler}. In addition, letting $a, b$ be sufficiently large, we can guarantee  that
\begin{align*}
&\|(\RR_1,\MM_1)\|_{L^1}\le \|(\vin,\bin)\|^2_{L^2}\le \delta_{2}\lambda^{-40\alpha}_1,\\
& \|(e^{t\Delta}\vin,e^{t\Delta}\bin)\|_{L^2} \le\|(\vin,\bin)\|_{L^2}\leq \delta_{2}^{1/2}<\delta^{1/2}_1,\\
& \|(\vin*\psi_{\ell_{0}},\bin*\psi_{\ell_{0}})\|_{H^3}\le \delta_{2}^{1/2}\ell^3_{0} \leq\lambda^5_1.
\end{align*}
For arbitrary smooth functions $e: [1/2,T]\to [0,\infty)$, $h:[1/2,T]\to(-\infty,\infty)$ satisfying the estimates \eqref{e:vin-energy-estimate}-\eqref{e:bin-helicity-estimate}, it is easy to verify that \eqref{e:energy-q-estimate}--\eqref{e:energy-h-estimate} for $q=1$  .

Then, making use of Proposition \ref{p:main-prop} inductively, we obtain a $L^2$ convergent sequence of functions $(v_q,b_q)\to (v,b)$ which solves \eqref{mhd}, with $\| v\|_{L^2}^2+\| b\|_{L^2}^2= e(t)$ and $\int_{\mathbb{T}^3}v\cdot b \dd x= h(t)$ for all $t\in[1,T]$. A standard  argument shows that $(v,b)\in C([0,T];L^2(\mathbb{T}^3))$, see \cite{zbMATH07038033} for more details. Moreover, from \eqref{e:vq-C1} and \eqref{e:velocity-diff}, there exists  $0<\epsilon\ll\beta$ such that $\{(v_q,~b_{q})\}$  is also a Cauchy sequence in $C_tH^{\epsilon}_x$ by interpolation. Thus, we obtain $(v,b)\in C_tH^{\epsilon}_x$.
\hfill $\qedsymbol$

The remainder of the paper is devoted to the proof of Proposition \ref{p:main-prop}.
\subsection{The proof sketch of Proposition \ref{p:main-prop}}
\label{ss:sketch-pf-main-prop}
Starting from a solution $(v_q, b_q, p_q, \RR_q, \MM_q)$ satisfying the estimates as in Proposition \ref{p:main-prop}, the broad scheme of the iteration is as follows.
\begin{enumerate}
    \item We defined $(v_{\ell_q},b_{\ell_q},p_{\ell_q},\RR_{\ell_q},\MM_{\ell_q})$ by mollification, and it is standard in convex integration schemes.
    \item We define a family of exact solutions $(\vex_l, \bex_l)_{l\ge0}$ to MHD by exactly solving the MHD system with initial data $(\vex_l, \bex_l)|_{t=t_l}= (v_{\ell_q}(t_l), b_{\ell_q}(t_l))$, where $t_l=l\tau_q$ defines an evenly spaced paritition of $[0,T]$.
    \item These solutions are glued together by a partition of unity, leading to the tuple $(\vv_q,\bb_q,\bar{p}_q,\RRR_q,\MMM_q)$. The stress error terms are zero when $t\in J_l,~l\geq 0$, see \cite{zbMATH07038033,2018Wild}.
    \item We define $(v_{q+1}, b_{q+1}) = (\vv_q+w_{q+1}, \bb_q+d_{q+1})$ by constructing a perturbation $(w_{q+1},d_{q+1})$. 
    \item Finally, we prove that the inductive estimates \eqref{e:vq-C0}--\eqref{e:energy-h-estimate} hold with $q$ replaced by $q+1$.
    \end{enumerate}
Step 4 is moderately involved and we breaks it into the following sub-steps : 
\begin{enumerate}
    \item For times $t\ge 1$, we use the `squiggling' cutoffs $\eta_l$ from \cite{zbMATH07038033,KMY} that allow energy to be added at such times, even outside the support of $(\RRR_q,\MMM_q)$, while cancelling a large part of $(\RRR_q,\MMM_q)$.
    \item For times $t< 1$, we instead employ the straight cutoffs introduced in \cite{Isett,2018Wild}. This ensures that $(\vv_q,\bb_q)|_{t=0}=(\vin * \psi_{\ell_{q-1}}* \psi_{\ell_{q}}, \bin * \psi_{\ell_{q-1}}* \psi_{\ell_{q}})$.
     \item Then we construct the seven parts of the perturbation by using the ``box flows''.

         $\bullet~$ Principal flows:~ $(\wpq,\dpq)$ plays a role in cancelling  the Reynolds and magnetic stresses $(\RRR_q,\MMM_q)$, while this would lead to some extra errors.

         $\bullet~$ Temporal flows:~ $(\wtq,\dtq)$ is used to cancel the extra errors which stem from $\div(\phi_{\bar{k}}\bar{k})$, where $\phi_{\bar{k}}$ is a traveling-wave.

         $\bullet~$ Inverse traveling wave flows:~ ($\wvq, \dvq$) is used to cancel the extra errors produced by $\div(\phi_{\bar{\bar{k}}}\bar{\bar{k}})$, where $\phi_{\bar{\bar{k}}}$ does not depend on $t$.

          $\bullet~$ Heat conduction flows :~ $(\wlq, \dlq)$ is used to
         cancel the extra errors producing by the inverse traveling wave flows ($\wvq, \dvq$).

         $\bullet~$ Corrector flows:~ $(\wcq,\dcq)$ is introduced to correct principal perturbation $(\wpq,\dpq)$ to enforce the incompressibility condition.

         $\bullet~$ Initial flows:~ $(\wsq,\dsq)$ can ensure that $(v_{q+1},b_{q+1})|_{t=0}=(\vin* \psi_{\ell_{q}}, \bin*\psi_{\ell_{q}})$. It should be noted that the above five types of flows are zero when $t=0$.

         $\bullet~$ Helicity flows:~ $(\whq,\dhq)$ makes the cross helicity satisfy  \eqref{e:energy-h-estimate} at $q+1$ level.
\end{enumerate}
It is noteworthy that the first five flows are enough to produce a weak solution $(v,b)$ for \eqref{mhd}, initial flows and helicity flows are used to control the helicity and show the non-uniqueness.

\section{Preliminary preparation of iteration}
In this section, we provide some preliminary preparation from $(v_q,b_q)$ to $(\vv_q,\bb_q)$, and it is essentially a modification as in \cite{zbMATH07038033,13Nonuniqueness,KMY}. For the sake of completeness, we briefly review relevant results in the process of constructing $(\vv_q,\bb_q)$ and readers can refer to \cite{zbMATH07038033,13Nonuniqueness,KMY} for the more details. {In Section 4, we will construct $(\vv_q,\bb_q)\rightarrow (v_{q+1},b_{q+1})$ to complete the iteration, which is the key ingredient of this paper.}
\label{s:construct}
\subsection{Mollification}

Let ${\ell_q}:= \lambda_q^{-6} $, and we define the functions $v_{\ell_q}, b_{\ell_q}$ and $\RR_{\ell_q}, \MM_{\ell_q}$ as follows:
\begin{align}
    v_{\ell_q} &\coloneq v_q * \psi_{\ell_q},
 &   \RR_{\ell_q} &\coloneq \RR_q * \psi_{\ell_q}  - (v_q \ootimes v_q) * \psi_{\ell_q}  + v_{\ell_q} \ootimes v_{\ell_q}+ (b_q \ootimes b_q) * \psi_{\ell_q}  - b_{\ell_q} \ootimes b_{\ell_q} , \label{e:v_ell}\\
     b_{\ell_q} &\coloneq b_q * \psi_{\ell_q},
 &   \MM_{\ell_q} &\coloneq \MM_q * \psi_{\ell_q}  - (v_q \otimes b_q) * \psi_{\ell_q}  + v_{\ell_q} \otimes b_{\ell_q}+ (b_q \otimes v_q) * \psi_{\ell_q} - b_{\ell_q} \otimes v_{\ell_q}. \label{e:v_ell}
\end{align}
Then, $(v_{\ell_q},b_{\ell_q},p_{\ell_q},\RR_{\ell_q},\MM_{\ell_q})$ satisfies the following equations
\begin{equation}
\left\{ \begin{alignedat}{-1}
\del_t v_{\ell_q}-\Delta v_{\ell_q} +\div (v_{\ell_q}\otimes v_{\ell_q})  +\nabla p_{\ell_q}   &= \div (b_{\ell_q}\otimes b_{\ell_q})+  \div \RR_{\ell_q} ,
\\
\del_t b_{\ell_q}-\Delta b_{\ell_q} +\div (v_{\ell_q}\otimes b_{\ell_q})   &= \div (b_{\ell_q}\otimes v_{\ell_q})+  \div \MM_{\ell_q} ,
\\
  \nabla \cdot v_{\ell_q} = 0,~~\nabla \cdot b_{\ell_q} &= 0,
\\
v_{\ell_q} |_{t=0}= \vin* \psi_{\ell_{q-1}}* \psi_{\ell_{q}} ,~~b_{\ell_q}|_{t=0} &=\bin* \psi_{\ell_{q-1}}* \psi_{\ell_{q}},
\end{alignedat}\right.  \label{e:mollified-euler}
\end{equation}
where $p_{\ell_q} \coloneq p_q *\psi_{\ell_q} -|v_q|^2 + |v_{\ell_q}|^2+|b_q|^2 - |b_{\ell_q}|^2,$ and we have used the identity $\div(fI_{3\times 3}) =\nabla f$ for a scalar field $f$. A simple computation shows the following mollification estimates: 
\begin{prop}[Estimates for mollified functions\cite{zbMATH07038033}]\label{p:estimates-for-mollified}
For any  $N \geq 0$, we have \footnote{Throughout this paper, we use the notation $x\lesssim y$ to denote $x\le Cy$, for a universal constant $C$ that may be different from line to line,  and $x\ll y$ to denote that $x$ is much less than $y$.}
\begin{align}
\|v_{\ell_q}-v_{q}\|_{L^2}+\|b_{\ell_q}-b_{q}\|_{L^2} &\lesssim \delta_{q+1}  \ell^{3\alpha}_{q}, \label{e:v_ell-vq}
\\
\|v_{\ell_q}\|_{H^{N+3}}+\|b_{\ell_q}\|_{H^{N+3}} &\lesssim  \lambda^5_{q}{\ell}^{-N}_q , \label{e:v_ell-CN+1}
\\
\|\RR_{\ell_q}\|_{W^{N,1}}+\|\MM_{\ell_q}\|_{W^{N,1}} &\lesssim \delta_{q+1}\ell^{3\alpha- N}_q  ,\label{e:R_ell}
\\
\Big|\int_{\mathbb{T}^{3}} (| v_{q}|^{2}-\left|v_{\ell_q}\right|^{2} )\dd x \Big|+\Big|\int_{\mathbb{T}^{3}} (| b_{q}|^{2}-\left|b_{\ell_q}\right|^{2}) \dd x \Big| &\lesssim \delta_{q+1} \ell^{3\alpha}_{q},\label{e:energy-v_ell}
\\
\Big|\int_{\mathbb{T}^{3}}
 v_{q}\cdot b_{q} -v_{\ell_q}\cdot b_{\ell_q} \dd x \Big|&\lesssim \delta_{q+1} \ell^{3\alpha}_{q}.\label{e:crossenergy-v_ell}
\end{align}
\end{prop}
\subsection{Classical exact  flows}
\label{ss:exact}
We define  $\tau_q$ and the  sequence of initial times $t_l$ ($l\in\mathbb N$) by
 \begin{align}
    \tau_q \coloneq  {\ell}^3_q\ll\|(v_{\ell_q},b_{\ell_q})\|^{-1}_{H^{\frac{5}{2}+\alpha}}, \qquad t_l\coloneq l\tau_q,   \label{e:tau_q-and-t_i}
\end{align}
and  $(\vex_l, \bex_l)$ denotes the unique strong solution to the following MHD system on $[t_l,t_{l+2}]$:
\begin{equation}\label{eulervi}
\left\{ \begin{alignedat}{-1}
\del_t \vex_l -\Delta \vex_l+\div (\vex_{l}\otimes \vex_{l})  +\nabla \pex_l   &=  \div (\bex_{l}\otimes \bex_{l}) ,
\\
\del_t \bex_{l}-\Delta \bex_l +\div (\vex_{l}\otimes \bex_{l})   &= \div (\bex_{l}\otimes \vex_{l})  ,
\\
  \div \vex_{l} = 0,~~\div \bex_{l} &= 0,
  \\
  \vex_l|_{t=t_l}= v_{\ell_q}(\cdot,t_l),~~\bex_l|_{t=t_l} &= b_{\ell_q}(\cdot,t_l).
\end{alignedat}\right.
\end{equation}

\begin{prop}[Estimates for classical solutions to MHD \cite{67Wu}]
\label{p:exact-euler}
Let $(v_0,b_0)\in H^{N_0}$ with $N_0\geq 3$, and $\div v_0=\div b_0=0$. Then there exists a unique local solution $(v,b)$ to \eqref{mhd} with $\nu_1=\nu_2=1$ satisfying
\[ \|(v(\cdot,t),b(\cdot,t))\|_{H^{N}} \lesssim  \|(v_0,b_0)\|_{H^{N}} ,~~N\in [\tfrac{5}{2},N_0], \]
where the local lifespan $T= \frac c{\|v_0\|_{H^{5/2+\alpha}}+\|b_0\|_{H^{5/2+\alpha}}}$ for some universal $c>0$.
\end{prop}
According to Proposition \ref{p:exact-euler},  the solvability of the Cauchy problem
 \eqref{eulervi} on $[t_{l},t_{l+2}]$ can be stated as:
\begin{cor}\label{estimate-0}
System \eqref{eulervi} possesses a unique local solution $(\vex_l, \bex_l)$ in $[t_l,t_{l+2}]$ such that
\begin{align}
&\|(\vex_l(\cdot,t), \bex_l(\cdot,t))\|_{L^{2}}\lesssim \|(v_{\ell_q},b_{\ell_q})\|_{L^{2}},\label{e:stability-1}\\
&\|(\vex_l(\cdot,t), \bex_l(\cdot,t))\|_{H^{N+3}}\lesssim \|(v_{\ell_q},b_{\ell_q})\|_{H^{N+3}}, \quad\quad\quad \forall N \geq 0,\label{e:stability-2}\\
   & \|(\vex_l -v_{\ell_q}, \bex_l -b_{\ell_q})\|_{H^{N}} \lesssim \tau_q \delta_{q+1} \ell_q^{-N-5/2+\alpha}, \quad~~~ \forall N \geq 0,\label{e:stability-3}
 \end{align}
 where $(\vex_l -v_{\ell_q},\bex_l -b_{\ell_q})$ has zero mean.
\end{cor}
\begin{proof}
\eqref{e:stability-1}-\eqref{e:stability-2} could be obtianed by Proposition \ref{p:exact-euler}. We want to  prove \eqref{e:stability-3}. Let $(v,b):=(\vex_l -v_{\ell_q}, \bex_l -b_{\ell_q})$, we have
\begin{equation}
\left\{ \begin{alignedat}{-1}
\del_t v -\Delta v+\div (\vex_{l}\otimes v+v\otimes v_{\ell_q})  +\nabla (\pex_l-p_{\ell_q})   &=  \div (\bex_{l}\otimes b+b\otimes b_{\ell_q})+ \div\RR_{\ell_q},
\\
\del_t b-\Delta b +\div (\vex_{l}\otimes b+v\otimes b_{\ell_q})   &= \div(\bex_{l}\otimes v+b\otimes v_{\ell_q} )+\div\MM_{\ell_q}  ,
  \\
  v|_{t=t_l}= 0,~~b|_{t=t_l} &= 0.
\end{alignedat}\right.
\end{equation}
When $N=0$, using the calsscial estimations in Besov space \cite{book} on $[t_l,t_{l+2}]$,  we deduce that
\begin{align}\label{sta1}
&\|(v, b)\|_{L^{\infty}_tB^0_{2,1}\cap L^2_tB^1_{2,1}\cap L^1_tB^2_{2,1}} \notag\\ \lesssim&
    \|\div (\vex_{l}\otimes v+v\otimes v_{\ell_q})\|_{L^1_tB^0_{2,1}}+\| \div (\bex_{l}\otimes b+b\otimes b_{\ell_q})\|_{L^1_tB^0_{2,1}} \notag\\
  &+\|\div (\vex_{l}\otimes b+v\otimes b_{\ell_q})\|_{L^1_tB^0_{2,1}}+\| \div(\bex_{l}\otimes v+b\otimes v_{\ell_q} )\|_{L^1_tB^0_{2,1}}
  +\|(\RR_{\ell_q},\MM_{\ell_q})\|_{L^1_tB^1_{2,1}}\notag\\
 \lesssim&\|(v, b)\|_{L^{\infty}_tB^0_{2,1}}\|(v_{\ell_q},\vex_{l},b_{\ell_q},\bex_{l})\|_{ L^{1}_tB^{5/2}_{2,1}}+\|(v, b)\|_{L^{2}_tB^1_{2,1}}\|(v_{\ell_q},\vex_{l},b_{\ell_q},\bex_{l})\|_{ L^{2}_tB^{3/2}_{2,1}}+\|(\RR_{\ell_q},\MM_{\ell_q})\|_{L^1_tB^1_{2,1}}\notag\\
 \lesssim& \tau_q\lambda^5_q\|(v, b)\|_{L^{\infty}_tB^0_{2,1}}
 +\tau^{1/2}_q\lambda^5_q\|(v, b)\|_{L^{2}_tB^1_{2,1}}+\tau_q \|(\RR_{\ell_q},\MM_{\ell_q})\|_{L^{\infty}_tB^1_{2,1}}\notag\\
 \lesssim&
 \tau_q \delta_{q+1} \ell_q^{-5/2+\alpha},
 \end{align}
where we use the fact that $\|(\RR_{\ell_q},\MM_{\ell_q})\|_{B^1_{2,1}}\lesssim
\|(\RR_{\ell_q},\MM_{\ell_q})\|_{B^{5/2}_{1,1}}\lesssim
\|(\RR_{\ell_q},\MM_{\ell_q})\|_{W^{5/2+\alpha,1}} \ll \delta_{q+1}\ell^{-5/2+\alpha}_q$.
When $N\geq 1$, similarly we deduce that
\begin{align}\label{sta2}
&\|(v, b)\|_{L^{\infty}_tB^N_{2,1}\cap L^2_tB^{N+1}_{2,1}\cap L^1_tB^{N+2}_{2,1}}\notag\\ \lesssim&
    \|\div (\vex_{l}\otimes v+v\otimes v_{\ell_q})\|_{L^1_tB^N_{2,1}}+\| \div (\bex_{l}\otimes b+b\otimes b_{\ell_q})\|_{L^1_tB^N_{2,1}} \notag\\
  &\|\div (\vex_{l}\otimes b+v\otimes b_{\ell_q})\|_{L^1_tB^N_{2,1}}+\| \div(\bex_{l}\otimes v+b\otimes v_{\ell_q} )\|_{L^1_tB^N_{2,1}}
  +\|(\RR_{\ell_q},\MM_{\ell_q})\|_{L^1_tB^{N+1}_{2,1}}\notag\\
 \lesssim&\|(v, b)\|_{L^{\infty}_tB^0_{2,1}}\|(v_{\ell_q},\vex_{l},b_{\ell_q},\bex_{l})\|_{ L^{1}_tB^{N+5/2}_{2,1}}+\|(v, b)\|_{L^{2}_tB^{N+1}_{2,1}}\|(v_{\ell_q},\vex_{l},b_{\ell_q},\bex_{l})\|_{ L^{2}_tB^{3/2}_{2,1}}+\|(\RR_{\ell_q},\MM_{\ell_q})\|_{L^1_tB^{N+1}_{2,1}}\notag\\
 \lesssim& \tau_q\ell^{-N}_q\lambda^5_q\|(v, b)\|_{L^{\infty}_tB^0_{2,1}}
 +\tau^{1/2}_q\lambda^5_q\|(v, b)\|_{L^{2}_tB^{N+1}_{2,1}}+\tau_q \|(\RR_{\ell_q},\MM_{\ell_q})\|_{L^{\infty}_tB^{N+1}_{2,1}}\notag\\
 \lesssim&
 \tau_q \delta_{q+1} \ell_q^{-N-5/2+\alpha}.
 \end{align}
Using the fact that $\|(\vex_l -v_{\ell_q}, \bex_l -b_{\ell_q})\|_{L^{\infty}_tH^{N}}\leq
\|(\vex_l -v_{\ell_q}, \bex_l -b_{\ell_q})\|_{ L^{\infty}_tB^N_{2,1}}$, \eqref{sta1} and \eqref{sta2} imply \eqref{e:stability-3}.
\end{proof}

\subsection{Gluing   flows}\label{ss:gluing}

Define the intervals $I_l,J_l$ ($l\ge 0)$ by \begin{align}
    I_l &\coloneq \Big[ t_l + \tfrac{\tau_q}3,\ t_l+\tfrac{2\tau _q}3\Big] ,
    \label{e:I_i-defn}\\
     J_l &\coloneq \Big(t_l - \tfrac{\tau_q}3,\ t_l+\tfrac{\tau _q}3\Big).
     \label{e:J_i-defn}
\end{align}
And $N_q$ denotes the smallest number so that
\[[0,T]\subseteq J_0 \cup I_0 \cup J_1 \cup I_1 \cup \dots \cup J_{N_q} \cup I_{N_q},\] i.e.
  \[N_q\coloneq\sup\Big\{ l\ge0 :   (J_l\cup I_l)\cap [0,T]\neq \emptyset \Big\}\le \left\lceil \frac T{\tau_q}\right\rceil.\]
For $N\ge 0$, let $\{ \chi_l\}_{l=1}^{N_q}$  be a partition of unity such that
$$\sum_{l=1}^{N_q}\chi_l(t)=1,~~~~t\in [-\tfrac{\tau_q}3,T+\tfrac{\tau _q}3],$$
where
\begin{align}
    & {\text{supp}}~\chi_l=I_{l-1} \cup J_l \cup I_{l},  &\chi_l |_{J_l} &=1,  &\|  \del_t^N \chi_l\|_{C^0_t} &\lesssim \tau_q^{-N},\ \   (N_q-1\geq l\geq 2), \label{e:chi_i-properties}\\
         & {\text{supp}}~\chi_1= J_0 \cup I_{0} \cup J_1 \cup I_{1}, &\chi_1 |_{[0,t_1] \cup J_1} &=1, &\|\del_t^N \chi_1\|_{C^0_t} &\lesssim \tau_q^{-N},\label{e:chi_1-properties}\\
         & {\text{supp}}~\chi_{N_q}=  I_{N_q-1} \cup J_{N_q} \cup I_{N_q}, &\chi_{N_q} |_{[t_{N_q-1},t_{N_q}] \cup J_{N_q}} &=1, &\|\del_t^N \chi_{N_q}\|_{C^0_t} &\lesssim \tau_q^{-N}. \label{e:chi_N-properties}
\end{align}
In particular, for $|l-j|\ge2$, $\text{supp}~\chi_l \cap \text{supp}~\chi_{j} =\emptyset$. Then we define the glued velocity, magnetic fields and pressure $(\vv_q, \bb_q, \ppp_q)$ by
\begin{align}
    \vv_q(x,t) &\coloneq \sum_{l=0}^{N_q-1} \chi_{l+1}(t) \vex_{l}(x,t) , \label{e:vv_q}\\
\bb_q(x,t) &\coloneq \sum_{l=0}^{N_q-1} \chi_{l+1}(t) \bex_{l}(x,t),\label{e:bb_q}
    \\
    \ppp_q(x,t) &\coloneq \sum_{l=0}^{N_q-1} \chi_{l+1}(t) \pex_l(x,t). \label{e:pp_q}
\end{align}
One can deduce that $$(\bar{v}_q(0,x),\bar{b}_q(0,x))=(v_{\ell_q}(0,x),b_{\ell_q}(0,x))=(\vin* \psi_{\ell_{q-1}}* \psi_{\ell_q},~\bin* \psi_{\ell_{q-1}}* \psi_{\ell_q}).$$ 

Furthermore, $(\vv_q,\bb_q)$ solves the following system for $t\in [0,T]$:
\begin{equation}\label{e:eqn-for-vex_i-v_ell}
\left\{ \begin{alignedat}{-1}
\del_t \vv_q-\Delta \vv_q +\div (\vv_q\otimes \vv_q)  +\nabla \pp_q   &=  \div (\bb_q\otimes \bb_q)+ \div\RRR_q,\\
\del_t \bb_q-\Delta \bb_q+\div (\vv_q\otimes \bb_q) &=\div (\bb_q\otimes \vv_q)+ \div\MMM_q,\\
\div \vv_q = 0,~~\div\bb_q &= 0,
  \\
  \vv_q |_{t=0}= v_{\ell_q}(\cdot,0),~~\bb_q|_{t=0} &= b_{\ell_q}(\cdot,0).
\end{alignedat}\right.
\end{equation}
Here $(\RRR_q, \MMM_q, \pp_q )$ is defined as follows:
\begin{align}
    \RRR_q\coloneq&
        ~\sum_{l=0}^{N_q}\del_t \chi_l \mathcal R(\vex_l-\vex_{l+1} ) - \chi_l(1-\chi_l)(\vex_l-\vex_{l+1} )\ootimes (\vex_l-\vex_{l+1} )\notag\\
        &+\chi_l(1-\chi_l)(\bex_l-\bex_{l+1} )\ootimes (\bex_l-\bex_{l+1} ), \label{RRR_q}\\
        \MMM_q\coloneq&
        ~\sum_{l=0}^{N_q}\del_t \chi_l \mathcal{R}_{a}(\bex_l-\bex_{l+1} ) - \chi_l(1-\chi_l)(\vex_l-\vex_{l+1} )\otimes (\bex_l-\bex_{l+1} )\notag\\
        &+\chi_l(1-\chi_l)(\bex_l-\bex_{l+1} )\otimes (\vex_l-\vex_{l+1} ), \label{MMM_q}\\
    \pp_q \coloneq &~\ppp_q - \sum_{l=0}^{N_q}\chi_l(1-\chi_l)\Big( |\vex_l - \vex_{l+1}|^2  - \int_{\mathbb T^3} |\vex_l - \vex_{l+1}|^2 \dd x\notag\\
    &-|\bex_l - \bex_{l+1}|^2  + \int_{\mathbb T^3} |\bex_l - \bex_{l+1}|^2 \dd x\Big),
\end{align}
where we have used the inverse divergence operators $\mathcal R$ and $\mathcal{R}_{a}$ in Section \ref{Geo}. It is easy to check that $\vex_l-\vex_{l+1}$, $\bex_l-\bex_{l+1}$ and $\pp_q$ have zero mean,  $\MMM_q$ is anti-symmetric, $\RRR_q$ is symmetric and trace-free. 

Combining Corollary \ref{estimate-0} with the estimates of solutions for the heat equation in periodic Besov spaces \cite{book}, we deduce the following two propositions after some computations:
\begin{prop}[Estimates for $(\vv_{q}-v_{\ell_q},\bb_{q}-b_{\ell_q})$]\label{estimate-vvq}
  For all $t\in[0,T]$ and $N\ge 0$, we have
\begin{align}\|(\vv_{q},\bb_{q}) \|_{H^{3+N}} & \lesssim   \lambda^5_{q} {\ell}^{-N}_q,\label{e:vv_q-bound}\\
\|(\vv_{q}-v_{\ell_q},\bb_{q}-b_{\ell_q})\|_{L^2} & \lesssim \delta_{q+1}^{1 / 2} {\ell}^\alpha_q,\label{e:stability-vv_q} \\
\|(\vv_{q}-v_{\ell_q},\bb_{q}-b_{\ell_q})\|_{H^N} & \lesssim \tau_{q} \delta_{q+1} {\ell}^{-N-5/2+\alpha}_q. \label{e:stability-vv_q-N}
\end{align}
\end{prop}


\begin{prop}[Estimates for $(\mathcal R(\vex_l-\vex_{l+1} ),~\mathcal{R}_{a}(\bex_l-\bex_{l+1})$]\label{estimate-1}
  For all $t\in[0,T]$ and  $ N\ge 0$, we have
\begin{align}
\|(\mathcal R(\vex_l-\vex_{l+1} ),\mathcal{R}_{a}(\bex_l-\bex_{l+1} ))\|_{W^{N,1}} & \lesssim \tau_{q} \delta_{q+1} {\ell}^{-N+\alpha}_q. \label{e:stability-vv_q-WN1}
\end{align}
\end{prop}

\begin{prop}[Estimates for $(\RRR_q,\MMM_q)$]For all $t\in [0,T]$ and  $\ N\ge 0$, we have
    \begin{align}
        \|(\RRR_q,\MMM_q)\|_{W^{N,1}}
        &\lesssim \delta_{q+1} {\ell}^{-N+\alpha}_q, \label{e:RRR_q-N+alpha-bd}
        \\
        \| (\del_t \RRR_q,\del_t \MMM_q)\|_{W^{N,1} }
        &\lesssim  \delta_{q+1} \tau^{-1}_q {\ell}^{-N+\alpha}_q.\label{e:matd-RRR_q}
    \end{align}
\end{prop}

\begin{prop}[Gaps between energy and helicity] For all $t\in[0,T]$, we have
\label{p:energy-of-vv_q}
\begin{align}
   \left| \int_{\mathbb T^3} (|\vv_q|^2+|\bb_q|^2-|v_{\ell_q}|^2-|b_{\ell_q}|^2)\dd x \right|\lesssim \delta_{q+1}\ell_q^{2\alpha},\label{p:energy-vv_q}\\
      \left| \int_{\mathbb T^3} (\vv_q\cdot\bb_q-v_{\ell_q}\cdot b_{\ell_q})\dd x \right|\lesssim \delta_{q+1}\ell_q^{2\alpha}.\label{p:crossenergy-vv_q}
\end{align}
\end{prop}


\subsection{Cutoffs}
\label{ss:defn-of-wq+1}
\subsubsection{Space-time cutoffs}
\label{sss:squiggling-cutoffs}
{To control the energy without impacting the initial data, we construct the following space-time cutoffs $\eta_l$ in this section.} We define the index
\[ N^0_q \coloneq  \bigg\lfloor \frac1{\tau_q} \bigg\rfloor - 2\in\mathbb \mathbb{N}^+, \]
and denote by $\{\eta_l\}_{l\geq1}$ the cutoffs such that:
\begin{align}
    \eta_l (x,t) \coloneq \begin{cases}
 \bar\eta_l(t) & 1\le l< N^0_q, \\
 \tilde\eta_l(x,t) &  N^0_q\le l\le N_q.
 \end{cases} \label{e:defn-cutoff-total}
\end{align}
We define $\bar\eta_l$ as in \cite{Isett,2018Wild} as follows: Let $\bar\eta_1\in C_c^\infty(J_1\cup I_1 \cup J_2;[0,1])$  satisfy
\[  \text{supp}~\bar\eta_1 = I_1 + \Big[  -\frac{\tau_q}6,\ \frac{\tau_q}6\Big] = \Big[  \frac{7\tau_q}6,\ \frac{11\tau _q}6\Big] , \]
be identically 1 on $I_1$, and possess the following estimates for $N\ge 0$:
\[ \|\del_t ^N \bar\eta_1\|_{C^0_t} \lesssim \tau_q^{-N}. \]
We set $\bar\eta_l(t)\coloneq\bar\eta_1(t-t_{l-1})$ for $1\le l< N^0_q$.

Next, we give the definition of $\tilde\eta_l(t,x)$ as in \cite{zbMATH07038033,KMY}. Let $\epsilon \in (0,1/3)$,  $\epsilon_0\ll 1$. For $N_q^0\le l\le N_q$, letting $\phi$ be a bump function such that ${\rm supp}\,\phi\subset [-1,1]$ and $\phi=1$ in $[-\tfrac12,\tfrac12]$, we define
\begin{align*}
     I_{l}'&\coloneq I_l + \Big[-\frac{(1-\epsilon)\tau_q}3, \frac{(1-\epsilon)\tau_q}3\Big] = \Big[l\tau_q+\frac{\epsilon\tau_q}3, l\tau_q+ \frac{(3-\epsilon)\tau_q}3\Big]  ,\\
         I_l'' &\coloneq \Big\{\big(x,t+\frac{2\epsilon\tau_q}3\sin (2\pi x_1) \big) : x\in\mathbb T^3, \ t \in I_l'\Big\}\subset  \mathbb T^3 \times \mathbb R,   \\
         \tilde \eta_l (x,t)&\coloneq \frac1{\epsilon_0^3}\frac1{\epsilon_0 \tau_q} \iint_{I_l''}\phi\Big(\frac{|x-y|}{\epsilon_0}\Big)\phi\Big(\frac{|t-s|}{\epsilon_0\tau_q}\Big)\dd y \dd s.
\end{align*}
\begin{figure}[H]
\begin{center}
\includegraphics[scale=0.65]{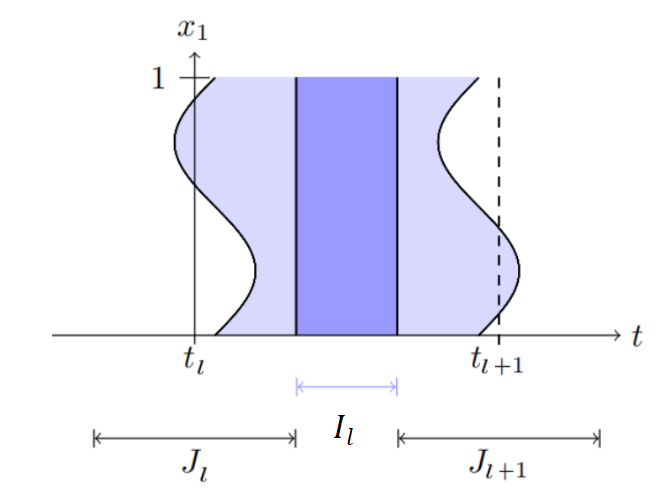}
\caption{The support of a single $\tilde\eta_l$. For each time $t\in[t_l,t_{l+1}]$, the integral $\int_0^1 \tilde\eta_l \dd x_1>c_\eta\approx 1/4$. Furthermore, $\{\text{{\rm supp} }\tilde\eta_l\}_{l\ge N^0_q}$ are pairwise disjoint sets. Figure from \cite{KMY}.}
\end{center}
\end{figure}
From the above discussion, we can obtain  the following lemma:
\begin{lem}[\cite{zbMATH07038033,KMY}] \label{eta}
For all $l=1,\dots,N_q$, The functions $\{\eta_l\}_{l\geq1}$ satisfy
    \begin{enumerate}
    \item $\eta_l \in C^\infty_c(\mathbb T^3 \times (J_{l}\cup I_{l}\cup J_{l+1}) ; [0,1])$ with:
        \begin{align}
         \|\del^n_t  \eta_l\|_{L^\infty_tC^m_x    }  \lesssim_{n,m} \tau_q^{-n},~~~n,m\ge 0. \label{e:eta-estimates}
    \end{align}
    \item $\eta_l(\cdot,t) \equiv 1$ for $t\in I_l$.
    \item $\rm{supp}~\eta_{l} \cap {\rm supp} ~\eta_j=\emptyset$ if $l\neq j$.
    \item  For all $t\in[t_{N^0_q},T]$, we have
    $$ c_\eta \le \sum_{l=0}^{N_q}\int_{\mathbb T^3}  \eta_l^2(x,t) \dd x \le 1$$
     for a fixed positive constant $c_\eta\approx\frac{1}{4}$ independent of $q$.
    \item For all $1\le     l<N^0_q$, $\eta_l $ only depends on $t$, and {\rm supp}~$\eta_l\subset  [  \tfrac{7\tau_q}6+(l-1)\tau_q,  \tfrac{11\tau _q}6+(l-1)\tau_q  ]$ .
\end{enumerate}
\end{lem}

\subsection{Helicity gap}
\label{ss:energy-gap}
First, for $t\in[0,T]$ we set new helicity gap such that
\begin{align}
h_q(t) \coloneq   \frac{1}3\Big(h(t) - \int_{\mathbb T^3} \vv_q\cdot\bb_q \dd x - \frac{\delta_{q+2}}{200} \Big). \label{h0:helicity-gap}
\end{align}
We deduce by Proposition \ref{p:estimates-for-mollified} and Proposition \ref{p:energy-of-vv_q} that  $h_q(t)$ is strictly positive in $[1-\tau_{q-1},T]$ and satisfies
\begin{align}\label{xingaps}
\tfrac{1}{400}\delta_{q+1}\leq 3h_q(t) \leq \tfrac{1}{90}\delta_{q+1},~~\forall t\in [1-\tau_{q-1},T].
\end{align}

Next, we define a function $\eta_{-1}:=\eta_{-1}(t)\in C_c^\infty([0,t_{N^0_q+1});[   0,1])$ such that
\begin{align}
\eta_{-1} \equiv 1,~~~0\le t\le t_{N^0_q},  \label{e:eta-1 property=1}
\end{align}
and  satisfies
\begin{align}
   \sup_t |\del_t^N \eta_{-1}(t)| \lesssim \tau_q^{-N},\quad~ N\ge 0. \label{e:eta-1-estimate}
\end{align}
Note that $t_{N^0_q+1}=\tau_q (\lfloor \tau_q^{-1}\rfloor -1) \le 1-\tau_q$, then $t\ge 1-\tau_q$ implies $\eta_{-1}(t)=0$.

The following figure shows the supports relationship between $\eta_{l}$ and $\eta_{-1}$ visually:
\begin{figure}[H]
\begin{center}
\includegraphics[scale=0.6]{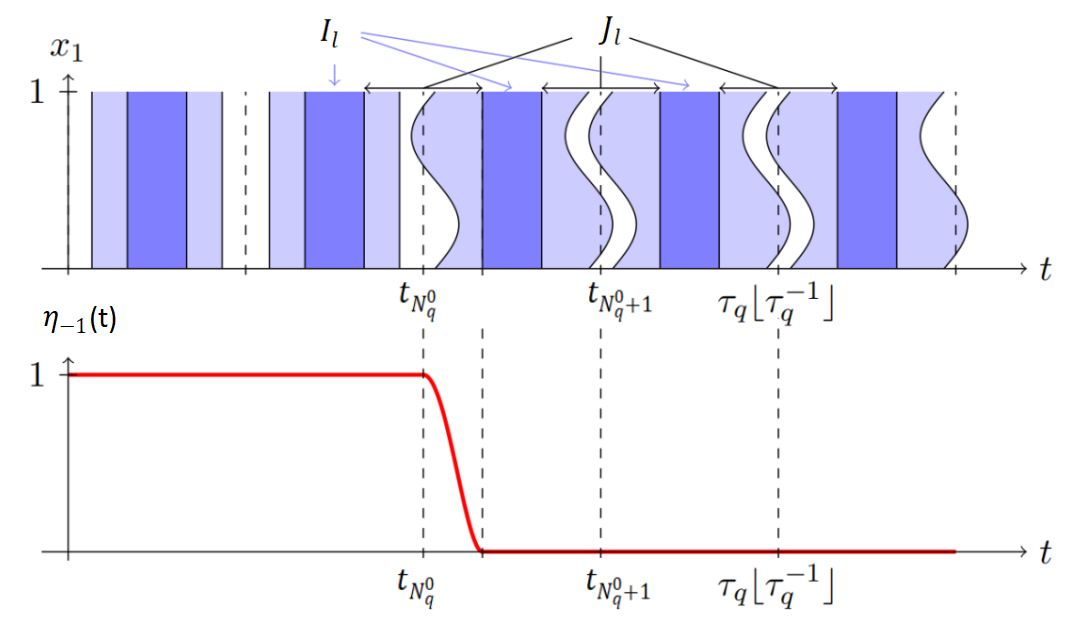}
\caption{ At time $t_{N^0_q}$, we switch from using straight cutoffs ${\bar{\eta}_l}$
to the squiggling cutoffs $\tilde\eta_l$. At time $t_{N^0_q+1} = \tau_q(
\big\lfloor  {\tau_q}^{-1}\big\rfloor-1) \leq 1-\tau_q  $, we start to control the helicity. Figure from \cite{KMY}.}
\end{center}
\end{figure}
Then, we modify the energy gap $h_q(t)$ by setting
\begin{align}
  h_{b,q} (t) &\coloneq \tfrac{\delta_{q+1}}{400}\aleph(t)+\tfrac{ h_q(t)(1-\aleph(t))}{ \eta_{-1}(t) + \sum_{l= 1}^{N_q} \int_{\mathbb T^3}\eta_l^2 ( x,t)\dd  x },  \label{h:energy-q-i-gap}
\end{align}
where $\aleph$ is a smooth cut-off function that is equal to $1$ for $t<1- \tau_{q-1}$ and $0$ for $t>1-\tau_q$ and satisfies $\del_t^N \aleph \lesssim \tau_{q-1}^{-N}\ll\tau_q^{-N}$. Using Lemma \ref{eta}, one can easily deduce that $\frac{\delta_{q+1}}{2400}\leq h_{b,q}(t)\leq \frac{1}{50}\delta_{q+1}$ for each $ t\in[0,T],$ which implies that $h_{b,q}$ is not much different from the original helicity gap in \eqref{e:energy-h-estimate}.

\section{Perturbation}
In this section, we construct the  perturbation $(w_{q+1}, d_{q+1})$ to iterate $(\vv_q,\bb_q)\rightarrow (v_{q+1},b_{q+1})$, which is the key ingredient of this paper.

\subsection{Stresses associated with the MHD system}
Let $(v_{q+1},b_{q+1})=(\vv_q+w_{q+1},\bb_q+d_{q+1})$. 
From \eqref{e:eqn-for-vex_i-v_ell}, we obtain the following MHD system with new Reynolds and magnetic stresses $(\RR_{q+1},\MM_{q+1})$:
\begin{equation}
\left\{ \begin{alignedat}{-1}
\del_t v_{q+1}-\Delta v_{q+1} + \div(v_{q+1} \otimes v_{q+1}-b_{q+1} \otimes b_{q+1})
     & = \div\RR_{q+1}-\nabla p_{q+1},\notag
 \\
 \del_t b_{q+1}-\Delta b_{q+1} + \div(v_{q+1} \otimes b_{q+1}-b_{q+1} \otimes v_{q+1})
     & = \div\MM_{q+1},\notag
  \\ (v_{q+1}, b_{q+1})|_{t=0}&=(\vin* \psi_{\ell_{q}},\bin* \psi_{\ell_{q}}),
\end{alignedat}\right.  \label{vq+1:subsol-mhd}
\end{equation}
where
\begin{align*}
p_{q+1}\coloneq& \pp_q(x,t)-\tfrac{1}{3}\tr[w_{q+1}  \otimes \vv_q+ \vv_q \otimes w_{q+1}-d_{q+1}  \otimes \bb_q- \bb_q \otimes d_{q+1}] - P_v ,\\
\RR_{q+1}=&\mathcal R[ \del_t w_{q+1}-\Delta w_{q+1}]+ [w_{q+1}  \mathring\otimes \vv_q+ \vv_q \mathring\otimes w_{q+1}-d_{q+1}  \mathring\otimes \bb_q- \bb_q \mathring\otimes d_{q+1}]  \\
&+\mathcal R[  \div (w_{q+1}\otimes w_{q+1}-d_{q+1}\otimes d_{q+1})+\div\RRR_q-  \nabla P_v]\notag\\
:=&\Rlinear+\Rosc,\\
\MM_{q+1}=&\mathcal{R}_a[ \del_t d_{q+1}-\Delta d_{q+1}]+ [w_{q+1}  \otimes \bb_q+ \vv_q \otimes d_{q+1}- \bb_q \otimes w_{q+1}-d_{q+1}  \otimes \vv_q]  \notag\\
&+\mathcal{R}_a\mathbb{P}_H[\  \div (w_{q+1}\otimes d_{q+1}-d_{q+1}\otimes w_{q+1})+\div\MMM_q] \\
:=&\Mlinear+\Mosc.
\end{align*}
Here we have used the fact that $\mathbb{P}_H\div A=\div A$ for any anti-symmetric matrix $A$. The definition of $P_v$ can be seen in \eqref{Pu}.

Before introducing the perturbation $(w_{q+1},d_{q+1})$, we provide two useful tools: ``geometric lemmas'' and ``box flows''.

\subsection{Two geometric lemmas}
\begin{lem}[First Geometric Lemma\cite{2Beekie}]\label{first L}
 Let $B_{\epsilon_b}(0)$ be a ball of radius $\epsilon_b$ centered at $0$ in the space of $3\times3$ skew-symmetric matrices. There exists a set $\Lambda_b\subset\mathcal{S}^2\cap\mathbb{Q}^3$ that consists of vectors $k$ with associated orthogonal basis $(k,\bar{k},\bar{\bar{k}}),~\epsilon_b>0$ and smooth positive functions $a_{b,k}:B_{\epsilon_b}(0)\rightarrow\mathbb{R}$, such that for $M \in B_{\epsilon_b}(0)$
$$M =\sum_{k\in\Lambda_b}a_{b,k}(M )(\bar{k}\otimes\bar{\bar{k}}-
\bar{\bar{k}}\otimes\bar{k}).$$
\end{lem}

\begin{lem}[Second Geometric Lemma\cite{2Beekie}]\label{first S}
 Let $B_{\epsilon_v}(\rm Id)$ be a ball of radius $\epsilon_v$ centered at $\rm Id$ in the space of $3\times3$ symmetric matrices. There exists a set ~$\Lambda_v\subset\mathcal{S}^2\cap\mathbb{Q}^3$ that consists of vectors $k$ with associated orthogonal basis $(k,\bar{k},\bar{\bar{k}}),~\epsilon_v>0$ and smooth positive functions $a_{v,k}:B_{\epsilon_v}(\rm Id)\rightarrow\mathbb{R}$,  such that for $R_u\in B_{\epsilon_v}(\rm Id)$
$$R=\sum_{k\in\Lambda_v}a_{v,k}(R)\bar{k}\otimes\bar{k}.$$
\end{lem}
To eliminate the helicity errors, we need another set $\Lambda_s$ which does not interact with $\Lambda_{v},\Lambda_{b}$, see Appendix \ref{Geo} for more details. For simplicity, we let $\Lambda:=\Lambda_{b}\cup\Lambda_{v}\cup\Lambda_{s}$.

\subsection{Box flows}
\label{ss:mikado}

In this section, let $\Phi:\mathbb{R}\rightarrow\mathbb{R}$ be a smooth cutoff function supported on the interval $[0,\tfrac{1}{8}]$. Assume that
$$\phi:=\frac{\dd^2}{\dd x^2}\Phi.$$
We define the stretching transformation as
$$r^{-\frac{1}{2}}\phi (N_{\Lambda}r^{-1}x),$$
where $r^{-1}$ is a positive integer number and $N_{\Lambda}$ is a large number such that $N_{\Lambda}{k},N_{\Lambda} \bar{k},N_{\Lambda}\bar{\bar{k}}\in \mathbb{Z}^3$. We periodize it as $\phi_r (x)$ on $[0,1]$.

Next, for lager positive integer numbers $r^{-1}, {\bar{r}}^{-1}, {\bar{\bar{r}}}^{-1}$, $\mu$ and $\sigma$, we set
$$\phi_{k}(x):=\phi_{r}(\sigma k\cdot x),~~\phi_{\bar{k}}(x,t):=\phi_{\bar{r}}( \sigma  \bar{k}\cdot x+\sigma \mu t),~~\phi_{\bar{\bar{k}}}(x):=\phi_{\bar{\bar{r}}}( \sigma \bar{\bar{k}}\cdot x ).$$
Then we define  a set of functions $\{\phi_{k,\bar{k},\bar{\bar{k}}} \}_{k\in\Lambda}:\TTT^3\times\R\to\R$ by
$$\phi_{k,\bar{k},\bar{\bar{k}}}(x,t)  := \phi_{k}(x-x_k)\phi_{\bar{k}}(x-x_k,t) \phi_{\bar{\bar{k}}}(x-x_k) ,~~k\in \Lambda, $$
where $x_k\in\R^3$ are shifts which guarantee that
\begin{equation}\label{suppkk'}
{\rm supp}~\phi_{k,\bar{k},\bar{\bar{k}}}\cap {\rm supp}~\phi_{k',\bar{k}',\bar{\bar{k}}'}=\emptyset,~~~{\rm if}~~k',k\in\Lambda,~k\neq k'.
  \end{equation}
There exist such shifts $x_k$ by the fact that $r,\bar{r}, \bar{\bar{r}}\ll 1$. \eqref{suppkk'} makes sense since $\phi_{k,\bar{k},\bar{\bar{k}}}$ supports in some small 3D boxes.  Readers can refer to \cite{1Cheskidov,zbMATH06710292,lyc} for this technique. In the rest of this paper, we still denote $\phi_{k}(x-x_k), \phi_{\bar{k}}(x-x_k,t)$ and $\phi_{\bar{\bar{k}}}(x-x_k)$ by $\phi_{k}(x)$, $\phi_{\bar{k}}(x,t)$ and $\phi_{\bar{\bar{k}}}(x)$, respectively.



Now, setting $\sigma=\lambda^{\frac{1}{128}}_{q+1},~{\mu}=\lambda^{\frac{17}{16}}_{q+1}
$ and $r=\bar{r}=\lambda_{q+1}^{-\frac{14}{16}},~\bar{\bar{r}}=\lambda_{q+1}^{-\frac{5}{16}}$.
 One can easily verify that $\phi_{k,\bar{k},\bar{\bar{k}}}$  has zero mean and we can deduce the following proposition:
\begin{prop}\label{guji1}
For $p\in[1,\infty]$,  we have
$$\|\phi_{k,\bar{k},\bar{\bar{k}}}\|_{L^p}\lesssim   {\lambda}^{\frac{14}{16}(1-\frac{2}{p})}_{q+1}
{\lambda}^{\frac{5}{16}(\frac{1}{2}-\frac{1}{p})}_{q+1},
~~~\big|{\rm{supp}}\,\phi_{k,\bar{k},\bar{\bar{k}}}\big|\approx \lambda^{-\frac{33}{16}}_{q+1}. $$
Moreover, we have $\|\phi_{k,\bar{k},\bar{\bar{k}}}\|_{L^2}=1$ after normalization.
\end{prop}

Finally, let $\Psi\in C^{\infty}(\mathbb{T})$ and $\psi=\frac{d^2}{dx^2}\Psi$. Set $\psi_k:=\psi(\lambda_{q+1}N_{\Lambda}kx)$ and normalize it such that $\|\psi_k\|_{L^2}=1$. One can deduce that
\begin{align}\label{wc0}
   \psi_k={\lambda^{-2}_{q+1}N^{-2}_{\Lambda}} \Delta[\Psi(\lambda_{q+1}N_{\Lambda}kx)]
   :={\lambda^{-2}_{q+1}N^{-2}_{\Lambda}}\Delta\Psi_k,~~k\in\Lambda.
\end{align}
It is easy to verify that $\psi_k
    \phi_{k,\bar{k},\bar{\bar{k}}}$ is also a $C^{\infty}_0(\mathbb{T}^3)$ function supported in 3D boxes, and we call $\psi_k
    \phi_{k,\bar{k},\bar{\bar{k}}}\bar{k}$ or $\psi_k
    \phi_{k,\bar{k},\bar{\bar{k}}}\bar{\bar{k}}$ \textbf{``box flows''} throughout this paper.

\subsection{Construction of perturbation}\label{flows}

Now, we introduce the following seven types of perturbation:\\
(1) To cancel the errors of the cross helicity, we construct a so-called ``helicity flow'' ($\whq$, $\dhq$). We are in position to define
\begin{align}
    \whq=\dhq &\coloneq \sum_{l;k\in\Lambda_{s}}  \eta_{l}h_{b,q}^{1/2}  \psi_k\phi_{k,\bar{k},\bar{\bar{k}}}\bar{\bar{k}} ,\label{e:defn-whq}
\end{align}
where  $h_{b,q}(t) $ comes from \eqref{h:energy-q-i-gap}.\\
(2) We aim to construct the principal corrector $(\wpq,\dpq)$ via geometric lemmas. Note that the geometric lemmas are valid for anti-symmetric matrices perturbed near $0$ matrix and symmetric matrices perturbed near $\rm Id$ matrix, we need the following smooth function $\chi$ introduced in \cite{2Beekie,luotianwen} such that the stresses are pointwise small.  More precisely, let $\chi:[0,\infty)\to \mathbb{R}^{+}$ be a smooth function satisfying
\begin{equation}
\chi(z)=\left\{ \begin{alignedat}{-1}
&1, \quad 0\le z\le 2,\\
&z, \quad z\ge 4.
\end{alignedat}\right.
\end{equation}
To cancel the magnetic stress $\MMM_q$, we construct the principal correctors $\wpqb$ and $\dpq$. Letting $\chi_{b}:=\chi\Big(\Big\langle\tfrac{M_b}{\delta_{q+1}\ell^{\alpha/2}_q }\Big\rangle\Big)$ and $M_b:=-\MMM_q$,~
$\rho_{b,q}:=\chi_{b}\delta_{q+1}\ell^{\alpha/3}_q$, we set that
\begin{align}
    \wpqb \coloneq & \sum_{l;k \in \Lambda_{b}}\eta_{l}\rho^{1/2}_{b,q} a_{b,k}\big(\tfrac{M_b}{ \rho_{b,q}}\big) \psi_k\phi_{k,\bar{k},\bar{\bar{k}}} \bar{k}
   \coloneq   \sum_{l;k \in \Lambda_{b}}a_{b,l,k}\psi_k\phi_{k,\bar{k},\bar{\bar{k}}}\bar{k}, \label{e:defn-wpq}  \\
    \dpq \coloneq  &\sum_{l;k \in \Lambda_{b }}  \eta_{l} \rho^{1/2}_{b,q}  a_{b,k}\big(\tfrac{M_b}{\rho_{b,q}}\big)  \psi_k\phi_{k,\bar{k},\bar{\bar{k}}} \bar{\bar{k}}
   \coloneq   \sum_{l;k \in \Lambda_{b}}a_{b,l,k}\psi_k\phi_{k,\bar{k},\bar{\bar{k}}}\bar{\bar{k}}.\label{e:defn-dpq}
\end{align}

Next, we want to construct the principal correctors $\wpqu$ to cancel the Reynolds stress $\RRR_q$. Let 
\begin{align}
\rho_q(t) \coloneq   \frac{1}3\Big(e(t) - \int_{\mathbb T^3} (|\vv_q|^2+|\bb_q|^2)\dd x-E(t) - \frac{\delta_{q+2}}{2} \Big), \label{e0:energy-gap}
\end{align}
where $E(t):=\int_{\mathbb{T}^3}(|\wpqb|^2+|\dpq|^2+|\whq|^2+|\dhq|^2)\dd x
.$
 One can easily deduce that $E(t)\leq \delta_{q+1}/10 $. Indeed,
 from the definition of $\phi_{k,\bar{k},\bar{\bar{k}}}$ and $\psi_{k}$, we deduce by  Lemma \ref{holder1} that
\begin{align}
|E(t)|\leq\|(\whq,\dhq,\wpqb,\dpq)\|^2_{L^2}&\lesssim (\|\eta_{l}h_{b,q}^{1/2}\|^2_{L^2}
+\|\eta_{l} \rho^{1/2}_{b,q}\|^2_{L^2}) \|\phi_{k,\bar{k},\bar{\bar{k}}}\|^2_{L^2}< \tfrac{\delta_{q+1}}{10}.
\end{align}
 Hence we deduce by Proposition \ref{p:estimates-for-mollified} and Proposition \ref{p:energy-of-vv_q} that $\rho_q(t)$ is strictly positive in $[1-\tau_{q-1},T]$ and satisfy
\begin{align}
\tfrac{1}{10}\delta_{q+1}\leq 3\rho_q(t) \leq 4\delta_{q+1}
\end{align}

Then, we modify the energy gap $\rho_q(t)$ by setting
\begin{align}
 \rho_{q,0} (t) &\coloneq \delta_{q+1}\aleph(t)+\tfrac{ \rho_q(t)(1-\aleph(t))}{ \eta_{-1}(t) + \sum_{l= 1}^{N_q} \int_{\mathbb T^3}\eta_l^2\chi_{v} (  x,t)\dd   x },  \label{e:energy-q-i-gap}
\end{align}
where $\chi_{v}:=\chi\Big(\Big\langle\tfrac{ R_v}{ \ell^{\alpha/4}_q\delta_{q+1}}\Big\rangle\Big)$ and $R_v:=\RRR_q-\sum_{ l;k\in \Lambda_{v }}a^2_{b,l,k}(\bar{k}\otimes\bar{k}
-\bar{\bar{k}}\otimes\bar{\bar{k}})$.

We firstly show that $\rho_{q,0}(t)$ is   well-defined. When $t\in I_i\cup  J_i$ with $ i\leq N_q^0$, we have $\eta_{-1}(t)=1$, $\rho_{q,0}(t)$ is  well-defined.  When $t\in I_i$ with $i\geq N_q^0$, we have
\begin{align}
 \int_{\mathbb T^3}\eta_l^2\chi_{v } (x,t)\dd x=\int_{\mathbb T^3}\chi_{v }\dd x=\int_{\mathbb T^3}\chi\Big(\Big\langle\tfrac{ R_v}{ \ell^{\alpha/4}_q\delta_{q+1}}\Big\rangle\Big)\dd x
\geq \int_{ \Big\langle\tfrac{ R_v}{ \ell^{\alpha/4}_q\delta_{q+1}}\Big\rangle\leq2 }1\dd x\
 \geq  \frac{1}{2},
 \end{align}
  where we use that $$\Big|m\Big\{x\big|\Big\langle\tfrac{ R_v}{ \ell^{\alpha/4}_q\delta_{q+1}}\Big\rangle\geq2\Big\}\Big|=\int_{ \Big\langle\tfrac{ R_v}{ \ell^{\alpha/4}_q\delta_{q+1}}\Big\rangle\geq2 } \dd x\leq\tfrac{\| R_v\|_{L^1}}{\ell^{\alpha/4}_q\delta_{q+1}}\ll\frac{1}{2}.$$
When $t\in J_i$ with $i\geq N_q^0$, we conclude that
\begin{align}\label{xybc}
\RRR_q=\MMM_q=0~~\text{and}~~ \Big<\tfrac{M_b}{\ell_{q}^{\alpha/2} \delta_{q+1}}\Big>,~\Big<\tfrac{R_v}{\ell_{q}^{\alpha/4} \delta_{q+1}}\Big>\leq2 .
 \end{align}
One could easily deduce  that $\chi_{b} =\chi_{v} =1$. Therefore,
  $$\int_{\mathbb T^3}\eta_l^2\chi_{v}  ( x,t)\dd   x=\int_{\mathbb T^3}\eta_l^2 ( x,t)\dd   x>c_{\eta}\approx\tfrac{1}{4}. $$
 Combining the definition of $\eta_{-1}$ with the above estimates shows that
\begin{equation}
\eta_{-1}(t) + \sum_{l= 1}^{N_q} \int_{\mathbb T^3}\eta_l^2\chi_{v,0} (  x,t)\dd x  \geq\left\{ \begin{alignedat}{-1}
1, & ~~~~0\le t\leq t_{N^0_q}, \\
\tfrac{1}{4}, & ~~~~t_{N^0_q}\leq t\leq T.
\end{alignedat}\right. \label{bufenmu:defn-cutoff-total}
\end{equation}
Therefore, we prove that $\rho_{q,0}$ is well-defined. Recalling the definitions of $\eta_{-1}(t),\aleph(t)$ and the fact that $1-\tau_{q-1}\leq t_{N^0_q} < t_{N^0_q+1}\leq1-\tau_{q}\leq t_{N^0_q+2}\leq 1$, we obtain that
$$\tfrac{\delta_{q+1}}{90}\leq \rho_{q,0}(t)\leq 10\delta_{q+1},~~~~\forall t\in[0,T].$$

 Now, setting
$\rho_{v,q}:= \rho_{q,0}\chi_v$, we construct the principal corrector $\wpqu$ such that
 \begin{align}
 \wpqu \coloneq\sum_{l;k \in \Lambda_{v}} \eta_{l}\rho_{v,q}^{1/2}  a_{v,k}\big({\rm Id}-\tfrac{R_v}{\rho_{v,q}}\big)\psi_k\phi_{k,\bar{k},\bar{\bar{k}}}\bar{k}
    \coloneq  \sum_{ l;k \in \Lambda_{v}}a_{v,l,k}\psi_k\phi_{k,\bar{k},\bar{\bar{k}}}\bar{k}.\label{e:defn-wpqu}
\end{align}
 Combining with \eqref{e:defn-wpq} and \eqref{e:defn-dpq}, we show the principal correctors $\wpq$ and $\dpq$ such that
\begin{equation}\label{pc}
\left\{ \begin{alignedat}{-1}
  \wpq
    \coloneq  & \sum_{ l;k \in \Lambda_{v}}a_{v,l,k}\psi_k\phi_{k,\bar{k},\bar{\bar{k}}}\bar{k}+\sum_{ l;k \in \Lambda_{b}}a_{b,l,k}\psi_k\phi_{k,\bar{k},\bar{\bar{k}}}\bar{k}
    =\wpqu+\wpqb, \\
    \dpq  \coloneq  & \sum_{ l;k \in \Lambda_{b}}a_{b,l,k}\psi_k\phi_{k,\bar{k},\bar{\bar{k}}}\bar{\bar{k}}.
\end{alignedat}\right.
\end{equation}

Collecting  Proposition
 \ref{estimate-vvq}--Proposition \ref{p:energy-of-vv_q} and using the same method as in the proof of Lemma~4.1--Lemma~4.4 in \cite{13Nonuniqueness}, we obtain:
\begin{prop}\label{a-estimate}
For $N\geq0$, we have
\begin{align}
& \|(a_{v,l,k},a_{b,l,k})\|_{L^{2}}\lesssim \delta^{1/2}_{q+1}{\ell}^{\frac{\alpha}{10}}_q ,\\
&\|(a_{v,l,k},a_{b,l,k})\|_{L^{\infty}}\lesssim {\ell}^{-3+\frac{\alpha}{10}}_q ,\\
&\|(a_{v,l,k},a_{b,l,k})\|_{H^{N}}\lesssim \delta^{1/2}_{q+1}{\ell}^{-5N+\frac{\alpha}{10}}_q ,\\
&\|\del_t(a_{v,l,k},a_{b,l,k})\|_{H^{N}}\lesssim \tau^{-1}_q\delta^{1/2}_{q+1}{\ell}^{-5N+\frac{\alpha}{10}}_q .
\end{align}
\end{prop}
We emphasize that $(a_{v,l,k},a_{b,l,k})$ oscillates at a  frequency $\ell^{-5}_q$, which have relatively small contribution comparing to the ``box flows''. Indeed, the ``box flows'' oscillate at a much higher frequency $\lambda^{\frac{1}{128}}_{q+1}\gg\ell^{-5}_q$ by taking the parameter $b$ sufficiently large in $\lambda_{q+1}$.

(3) {Because the supports of the ``box flows'' are much smaller than the supports of the Mikado flows, we choose the ``box flows'' instead of Mikado flows. However, this choice gives rise to extra errors in the oscillation terms, since they are not divergence-free. To overcome this difficulty, 
we introduce temporal type flows $\wtq,\dtq$  and inverse traveling wave flows $\wvq,\dvq$ , which help us eliminate the extra errors.} Now, let us define the temporal flows $\wtq$ and $\dtq$  by
\begin{align}
\wtq:&=-\mathbb{P}_{H}\mathbb{P}_{>0}\sum_{l;k\in\Lambda_{v}}\tfrac{1}{{\mu}}a^2_{v,l ,k}
\phi^2_{k,\bar{k},\bar{\bar{k}}}
\bar{k}-\mathbb{P}_{H}\mathbb{P}_{>0}\sum_{l;k\in\Lambda_{b}}\tfrac{1}{{\mu}}a^2_{b,l ,k}
\phi^2_{k,\bar{k},\bar{\bar{k}}}
\bar{k},
\label{e:defn-wtq}\\
\dtq:&=-\mathbb{P}_{H}\mathbb{P}_{>0}\sum_{l;k\in\Lambda_{b}}
\tfrac{1}{{\mu}}a^2_{b,l,k}
\phi^2_{k,\bar{k},\bar{\bar{k}}}
\bar{\bar{k}}.
\label{e:defn-dtq}
\end{align}

(4) We construct  the inverse traveling wave flows $\wvq, \dvq$ such that
\begin{align}
\wvq:&=(\mu\sigma)^{-1}\mathbb{P}_{H}\mathbb{P}_{>0}
\sum_{l;k\in\Lambda_{b}}  a^2_{b,l,k}
\div\big( ({\del^{-1}_t\mathbb{P}_{>0}\phi^2_{{\bar{k}}})}\phi^2_{k}\phi^2_{\bar{\bar{k}}}
\bar{\bar{k}}\otimes\bar{\bar{k}}\big),
\label{e:defn-wvq}\\
\dvq:&=(\mu\sigma)^{-1}\mathbb{P}_{H}\mathbb{P}_{>0}\sum_{l;k\in\Lambda_{b}}
 a^2_{b,l,k}
\div\big((\del^{-1}_t\mathbb{P}_{>0}\phi^2_{{\bar{k}}})\phi^2_{k}\phi^2_{\bar{\bar{k}}}
\bar{\bar{k}}\otimes\bar{k} \big),
\label{e:defn-dvq}
\end{align}
where $\del^{-1}_t\mathbb{P}_{>0}\phi^2_{{\bar{k}}}:=\int_{0}^{\sigma(\bar{k}(x-x_k)+\mu t)}(\phi^2_r(z)-1) \dd z$. Since $\phi^2_r(z)-1$ has zero mean, we have $|\del^{-1}_t\mathbb{P}_{>0}(\phi^2_{{\bar{k}}})|\leq2$.

(5) In order to match the inverse traveling wave flows $\wvq, \dvq$, we need to construct  the following heat conduction flows $\wlq, \dlq$ :
\begin{align}
\wlq: =& (\mu\sigma)^{-1}\mathbb{P}_{H}\mathbb{P}_{>0}
\sum_{l;k\in\Lambda_{b}}  a^2_{b,l,k}
\int_{0}^{t}e^{(t-\tau)\Delta}\Delta\div\big( (\del^{-1}_t\mathbb{P}_{>0} \phi^2_{{\bar{k}}})\phi^2_{k}\phi^2_{\bar{\bar{k}}}
\bar{\bar{k}}\otimes\bar{\bar{k}}\big)\dd\tau
\notag\\
&+\mathbb{P}_{H}\mathbb{P}_{>0}\sum_{l;k\in\Lambda_{b}}  a^2_{b,l,k}
\int_{0}^{t}e^{(t-\tau)\Delta}\div\big(\phi^2_{k}\phi^2_{\bar{\bar{k}}}
\bar{\bar{k}}\otimes\bar{\bar{k}}\big)\dd\tau,
\label{e:defn-wlq}\\
\dlq:=& (\mu\sigma)^{-1}\mathbb{P}_{H}\mathbb{P}_{>0}\sum_{l;k\in\Lambda_{b}}
 a^2_{b,l,k}
\int_{0}^{t}e^{(t-\tau)\Delta}\Delta\div\big((\del^{-1}_t\mathbb{P}_{>0}\phi^2_{{\bar{k}}})
\phi^2_{k}\phi^2_{\bar{\bar{k}}}
\bar{\bar{k}}\otimes\bar{k} \big)\dd\tau\notag\\
&+\mathbb{P}_{H}\mathbb{P}_{>0}\sum_{l;k\in\Lambda_{b}}  a^2_{b,l,k}
\int_{0}^{t}e^{(t-\tau)\Delta}\div\big(\phi^2_{k}\phi^2_{\bar{\bar{k}}}
\bar{\bar{k}}\otimes {\bar{k}}\big)\dd\tau.
\label{e:defn-dlq}
\end{align}

(6) Since $\wtq,\wlq,\wvq$ and $ \dtq, \dlq, \dvq $ are divergence-free, it suffices to define two small correctors $\wcq,\dcq$ such that
$$\div(\whq+\wpq+\wcq)=0, \quad \div(\dhq+\dpq+\dcq)=0.$$ 
Firstly, noting that $\psi_{k}\bar{k}$ and $\psi_{k}\bar{\bar{k}}$  are zero mean and divergence-free, we have by \eqref{wc0}
$$\psi_{k}\bar{k}
=-\frac{\curl\curl}{\Delta}  \Big[\Delta\tfrac{\Psi_k\bar{k}}{\lambda^2_{q+1}N^2_{\Lambda}}\Big]
:=\frac{\curl}{\lambda_{q+1}}F_{\bar{k}}~~and~~\psi_{k}\bar{\bar{k}}
=-\frac{\curl\curl}{\Delta}  \Big[\Delta\tfrac{\Psi_k\bar{\bar{k}}}{\lambda^2_{q+1}N^2_{\Lambda}}\Big]:=\frac{\curl}{\lambda_{q+1}}F_{\bar{\bar{k}}}.$$
Hence,
$$
 \wpq+\whq=\sum_{l;k \in \Lambda}\frac{a_{v,c,1}}{\lambda_{q+1}}\curl F_{\bar{k}}+
\frac{a_{v,c,2}}{\lambda_{q+1}}\curl F_{\bar{\bar{k}}},
$$
where $a_{v,c,1}= (a_{v,l,k}\mathbf{1}_{k\in\Lambda_v}+ a_{b,l, k}\mathbf{1}_{k\in\Lambda_{b }})\phi_{k,\bar{k},\bar{\bar{k}}}$ and $a_{v,c,2}=   \eta_{l}h^{1/2}_{b,q}\mathbf{1}_{k\in\Lambda_{s}}
 \phi_{k,\bar{k},\bar{\bar{k}}} $.

Next, we set
\begin{align}
    \wcq \coloneq
   \sum_{l ;k \in \Lambda}\tfrac{1}{\lambda_{q+1}}\nabla a_{v,c,1}\times F_{\bar{k}}+\tfrac{1}{\lambda_{q+1}}\nabla a_{v,c,2}\times F_{\bar{\bar{k}}}.\label{e:defn-wcq}
\end{align}
This equality leads to $$\wpq+\whq+\wcq= \sum_{l ;k \in \Lambda}\curl \Big(\tfrac{a_{v,c,1}}{\lambda_{q+1}}F_{{\bar{k}}}+\tfrac{a_{v,c,2}}{\lambda_{q+1}}F_{{\bar{\bar{k}}}}\Big)$$ by $$\curl (fW) = \nabla f \times W + f\curl W.$$

Similarly, letting
$$a_{b,c}= (a_{b,l ,k} \mathbf{1}_{k\in\Lambda_{b }}
+  \eta_{l}h^{1/2}_{b,q}\mathbf{1}_{k\in\Lambda_{s}}) \phi_{k,\bar{k},\bar{\bar{k}}},$$
 we define
\begin{align}
    \dcq &\coloneq\sum_{l ;k \in \Lambda }\tfrac{1}{\lambda_{q+1}}\nabla a_{b,c}\times F_{\bar{\bar{k}}}.
\end{align}
One deduces that $$\wcq+\wpq+\whq~~~\text{and}~~~ \dcq+\dpq+\dhq$$ are divergence-free and have zero mean.

(7) Finally, we define the following ``initial flows'':
\begin{align}
    \wsq &\coloneq e^{t\Delta}(\vin* \psi_{\ell_q}-\vin* \psi_{\ell_{q-1}}* \psi_{\ell_q})  \label{e:defn-v_q+1-and-w_s},\\
    \dsq &\coloneq e^{t\Delta}(\bin* \psi_{\ell_q}-\bin* \psi_{\ell_{q-1}}* \psi_{\ell_q} ) \label{b:defn-v_q+1-and-w_s}.
\end{align}
From the definition of $\eta_l$ in \eqref{e:defn-cutoff-total}, one obtains that
 $$(\whq+\wpq+\wtq+\wvq+\wlq+\wcq )(0,x)=0, \,\,(\dhq+\dpq+\dtq+\dvq+\dlq+\dcq ) (0,x)=0. $$
 Hence, \eqref{e:defn-v_q+1-and-w_s} and \eqref{b:defn-v_q+1-and-w_s} guarantee that $v_{q+1}(0,x)=\vin* \psi_{\ell_q}$ and $b_{q+1}(0,x)=\bin* \psi_{\ell_q}$. This fact implies that \eqref{e:RR_q-C1} holds with $q$ replaced by $q+1$. Moreover, one can easily deduce that $\wsq$ and $\dsq$ are divergence-free and have zero mean.

To sum up, we construct
\begin{align}
    w_{q+1} &\coloneq \whq+\wpq+\wtq+\wlq+\wvq+\wcq+\wsq,    \label{e:defn-v_q+1-and-w_q+1}\\
    d_{q+1} &\coloneq \dhq+\dpq+\dtq+\dlq+\dvq+\dcq +\dsq,    \label{e:defn-b_q+1-and-d_q+1}
\end{align}
which help us finish the iteration from $(\vv_q,\bb_q)$ to $(v_{q+1},b_{q+1})$.


Now, we show that \eqref{e:vq-C0} and \eqref{e:vq-C1} hold at $q+1$ level.
\begin{prop}[Estimates for $w_{q+1}$ and $d_{q+1}$]
\label{c:estimates-for-wpq-dpq}
\begin{align}
\|(\whq,\dhq)\|_{L^2}+\tfrac1{\lambda^5_{q+1}} \|(\whq,\dhq)\|_{H^3} &\le \tfrac{1}{10}\delta_{q+1}^{1/2},\label{estimation-whq}\\
\|(\wpq,\dpq)\|_{L^2}+\tfrac1{\lambda^5_{q+1}} \|(\wpq,\dpq)\|_{H^3} &\le 100\delta_{q+1}^{1/2},\label{estimation-wpq}\\
      \|(\wtq,\dtq)\|_{L^2}+\tfrac1{\lambda^5_{q+1}} \|(\wtq,\dtq)\|_{H^3} &\le \lambda^{-50\alpha}_{q+1}\delta_{q+2},\label{estimation-wtq}\\
      \|(\wvq,\dvq)\|_{L^2}+\tfrac1{\lambda^5_{q+1}} \|(\wvq,\dvq)\|_{H^3} &\le \lambda^{-50\alpha}_{q+1}\delta_{q+2},\label{estimation-wvq}\\
      \|(\wcq,\dcq)\|_{L^2}+\tfrac1{\lambda^5_{q+1}} \|(\wcq,\dcq)\|_{H^3} &\le \lambda^{-50\alpha}_{q+1} \delta_{q+2},\label{estimation-wcq}\\
       \|(\wlq,\dlq)\|_{L^2}+\tfrac1{\lambda^5_{q+1}} \|(\wlq,\dlq)\|_{H^3} &\le \lambda^{-50\alpha}_{q+1}\delta_{q+2},\label{estimation-wlq}\\
        \|(\wsq,\dsq)\|_{L^2}+\tfrac1{\lambda^5_{q+1}} \|(\wsq,\dsq)\|_{H^3} &\le  \lambda^{-50\alpha}_{q+1}\delta_{q+2},\label{estimation-wsq}\\
            \|(w_{q+1},d_{q+1})\|_{L^2}+\tfrac1{\lambda^5_{q+1}} \|(w_{q+1},d_{q+1})\|_{H^3} &\le 400 \delta_{q+1}^{1/2}.\label{wq+L2}
\end{align}
\end{prop}
\begin{proof}
Firstly, from the definition of $\phi_{k,\bar{k},\bar{\bar{k}}}$ and $\psi_{k}$, we deduce by Proposition \ref{guji1} ,~Proposition \ref{a-estimate} and Lemma~\ref{holder1}  that
\begin{align}
\|(\whq,\dhq)\|_{L^2}&\leq \|\eta_{l}h_{b,q}^{1/2}  \psi_{k}\phi_{k,\bar{k},\bar{\bar{k}}}\|_{L^2}\lesssim \|\eta_{l}h_{b,q}^{1/2}\|_{L^2}  \|\phi_{k,\bar{k},\bar{\bar{k}}}\|_{L^2}\leq \tfrac{\delta^{1/2}_{q+1}}{20},\label{wpdpL2}
\end{align}
and
\begin{align}
\|(\whq,\dhq)\|_{{H}^3}&\leq \lambda^5_{q+1}\tfrac{\delta^{1/2}_{q+1}}{20}.\label{wpdpH3}
\end{align}
Hence, \eqref{wpdpL2} and  \eqref{wpdpH3} imply \eqref{estimation-whq}.  In the same way as deriving \eqref{wpdpL2} and \eqref{wpdpH3}, we  obtain \eqref{estimation-wpq}--\eqref{estimation-wcq}.For example, to prove \eqref{estimation-wvq}, recall that:
\begin{align}
\wvq:&=(\mu\sigma)^{-1}\mathbb{P}_{H}\mathbb{P}_{>0}
\sum_{l;k\in\Lambda_{b}}  a^2_{b,l,k}
\div\big( ({\del^{-1}_t\mathbb{P}_{>0}\phi^2_{{\bar{k}}})}\phi^2_{k}\phi^2_{\bar{\bar{k}}}
\bar{\bar{k}}\otimes\bar{\bar{k}}\big)
\end{align}
where $\del^{-1}_t\mathbb{P}_{>0}\phi^2_{{\bar{k}}}:=\int_{0}^{\sigma(\bar{k}(x-x_k)+\mu t)}(\phi^2_r(z)-1) \dd z$. Since $\phi^2_r(z)-1$ has zero mean, we have $|\del^{-1}_t\mathbb{P}_{>0}(\phi^2_{{\bar{k}}})|\leq2$. We deduce that
\begin{align}
\|\wvq\|_{L^2} \lesssim &(\mu\sigma)^{-1}\|a^2_{b,l,k}\|_{L^{\infty}}\| ({\del^{-1}_t\mathbb{P}_{>0}\phi^2_{{\bar{k}}})}\phi^2_{k}\div\big(\phi^2_{\bar{\bar{k}}}
\bar{\bar{k}}\otimes\bar{\bar{k}}\big)\|_{L^2}\notag\\
\lesssim &(\mu\sigma)^{-1}\ell^{-15}_q(\sigma \bar{r}^{-1/2}\bar{\bar{r}}^{-3/2} )
\lesssim \ell^{-15}_q\lambda^{-\frac{17}{16}+\frac{7}{16}+\frac{15}{32}}_{q+1}
\ll\lambda^{-50\alpha}_{q+1}\delta_{q+2},\notag\\
\|\wvq\|_{H^3} \lesssim &\lambda^{3}_{q+1}(\mu\sigma)^{-1}\ell^{-15}_q(\sigma \bar{r}^{-1/2}\bar{\bar{r}}^{-3/2} )
\lesssim \ell^{-15}_q\lambda^{3-\frac{17}{16}+\frac{7}{16}+\frac{15}{32}}_{q+1}
\ll\lambda^5_{q+1}\lambda^{-50\alpha}_{q+1}\delta_{q+2},
\end{align}
which implies $\| \wvq \|_{L^2}+\tfrac1{\lambda^5_{q+1}} \| \wvq \|_{H^3}  \le \lambda^{-50\alpha}_{q+1}\delta_{q+2}$. Similarly for $\dvq$, one can easily deduce \eqref{estimation-wvq}.


{Secondly}, we aim to prove \eqref{estimation-wlq}. For simplicity, we denote $\wlq$ by
\begin{align}
\wlq: =& (\mu\sigma)^{-1}\mathbb{P}_{H}\mathbb{P}_{>0}
\sum_{l;k\in\Lambda_{b}}  a^2_{b,l,k}
\int_{0}^{t}e^{(t-\tau)\Delta}\Delta g(x,t)\dd\tau
\notag\\
&+\mathbb{P}_{H}\mathbb{P}_{>0}\sum_{l;k\in\Lambda_{b}}  a^2_{b,l,k}
\int_{0}^{t}e^{(t-\tau)\Delta}\div\big(h({\sigma}k\cdot x, {\sigma}\bar{\bar{k}}\cdot x)
\bar{\bar{k}}\otimes\bar{\bar{k}}\big)\dd\tau,
\end{align}
where  $h({\sigma}k\cdot x, {\sigma}\bar{\bar{k}}\cdot x)=\phi^2_{k}\phi^2_{\bar{\bar{k}}}=\phi^2_{r}(\sigma k\cdot (x-x_k))\phi^2_{r}(\sigma {\bar{\bar {k}}}\cdot (x-x_k))$ and
$$g(t,x):=\div\big( (\del^{-1}_t\mathbb{P}_{>0}\phi^2_{{\bar{k}}})\phi^2_{k}\phi^2_{\bar{\bar{k}}}
\bar{\bar{k}}\otimes\bar{\bar{k}}\big)= (\del^{-1}_t\mathbb{P}_{>0}\phi^2_{{\bar{k}}})\phi^2_{k}\div\big(\phi^2_{\bar{\bar{k}}}
\bar{\bar{k}}\otimes\bar{\bar{k}}\big).$$
 Using  the estimates of solution for the heat equation in \cite{book}, we deduce that for $0<\alpha\leq\min\{\frac{1}{b^6},\frac{\beta}{b^3}\}$,
\begin{align}\label{wlq1}
\big\|\int_{0}^{t}e^{(t-\tau)\Delta}\Delta g \dd\tau\big\|_{L^{\infty}_tL^2}
 \leq\|g   \|_{{L}^{\infty}_tH^{\alpha}}\lesssim  \lambda^{\alpha}_{q+1}\sigma r^{-1/2}\bar{\bar{r}}^{-3/2}.
\end{align}
Similarly, noting that
$ \div \big(h({\sigma}k\cdot x, {\sigma}\bar{\bar{k}}\cdot x)\bar{\bar{k}}\otimes\bar{\bar{k}}\big)={\sigma}(\partial_2h)
({\sigma}k\cdot x, {\sigma}\bar{\bar{k}}\cdot x)\bar{\bar{k}}$, we deduce that
\begin{equation}\label{estimate-h}
\begin{aligned}
&\big\|\int_{0}^{t}e^{(t-\tau)\Delta} \div\big(h({\sigma}k\cdot x, {\sigma}\bar{\bar{k}}\cdot x) \bar{\bar{k}}\otimes\bar{\bar{k}}\big) \dd\tau\big\|_{L^{\infty}_tL^2}\\
\lesssim&{\sigma}\big\|\tfrac{(\partial_2h)
({\sigma}k\cdot x,{\sigma}\bar{\bar{k}}\cdot x)\bar{\bar{k}}}{\Delta}\big\|_{L^{\infty}_tH^{\alpha}}
=\sigma^{-1}\big\|\tfrac{\partial_2h}{\Delta}({\sigma}k\cdot x, {\sigma}\bar{\bar{k}}\cdot x)\big\|_{L^{\infty}_tH^{\alpha}},
\end{aligned}
\end{equation}
where we have used the fact that
$$(\partial_2h)({\sigma}k\cdot x, {\sigma}\bar{\bar{k}}\cdot x)=\sigma^{-2}\Delta[\tfrac{\partial_2h}{\Delta}({\sigma}k\cdot x, {\sigma}\bar{\bar{k}}\cdot x)],~~(k,{\bar{k}},\bar{\bar{k}})\in\Lambda.$$
Since $h(\cdot,\cdot)\in C^{\infty}(\mathbb{T}^2)$ for fixed $q$ and $k\perp \bar{\bar{k}}$, we have for $\alpha>0$,
 $$\|\tfrac{\partial_2h}{\Delta}({\sigma}k\cdot x, {\sigma}\bar{\bar{k}}\cdot x)\|_{L^2(\mathbb{T}^3)}=\|\tfrac{\partial_2h}{\Delta}(\cdot,\cdot)\|_{L^2(\mathbb{T}^2)}\lesssim \|h\|_{W^{\alpha ,1}(\mathbb{T}^2)}$$
and
$$\|\tfrac{\partial_2h}{\Delta}({\sigma}k\cdot x, {\sigma}\bar{\bar{k}}\cdot x)\|_{\dot{H}^1(\mathbb{T}^3)}\lesssim\sigma\|h\|_{L^2(\mathbb{T}^2)}\lesssim \sigma \|h \|_{W^{1+\alpha ,1}(\mathbb{T}^2)}.$$
Plugging the above two estimates into \eqref{estimate-h} yields that
\begin{equation}\label{wlq2}
\begin{aligned}
&\big\|\int_{0}^{t}e^{(t-\tau)\Delta} \div\big(h({\sigma}k\cdot x, {\sigma}\bar{\bar{k}}\cdot x) \bar{\bar{k}}\otimes\bar{\bar{k}}\big) \dd\tau\big\|_{L^{\infty}_tL^2}\\
\lesssim&  \sigma^{-1}\|\tfrac{\partial_2h}{\Delta}({\sigma}k\cdot x, {\sigma}\bar{\bar{k}}\cdot x)\|^{1-\alpha}_{L^2(\mathbb{T}^3)}\|\tfrac{\partial_2h}{\Delta}({\sigma}k\cdot x, {\sigma}\bar{\bar{k}}\cdot x)\|^{\alpha}_{\dot{H}^1(\mathbb{T}^3)}\\
\lesssim & \sigma^{-1}\|h\|^{1-\alpha}_{W^{\alpha ,1}} \|h \|^{\alpha}_{W^{1+\alpha ,1}} \lesssim\sigma^{-1+\alpha}\lambda^{2\alpha}_{q+1}.
\end{aligned}
\end{equation}
Combining \eqref{wlq1} with \eqref{wlq2}, we have
$$\|\wlq\|_{L^\infty_tL^2}\lesssim  \ell^{-6}_{q}\lambda^{4\alpha}_{q+1}(\mu^{-1}r^{-1/2}\bar{\bar{r}}^{-3/2}+\sigma^{-1})
\lesssim  \ell^{-6}_{q}(\lambda_{q+1}^{-\frac{17}{16}+\frac{7}{16}+\frac{15}{32}+2\alpha} +\lambda^{-\frac{1}{128}+2\alpha}_{q+1})\ll \lambda^{-50\alpha}_{q+1}\delta_{q+2},$$
where $\sigma=\lambda^{\frac{1}{128}}_{q+1},~{\mu}=\lambda^{\frac{17}{16}}_{q+1}
$ and $r=\bar{r}=\lambda_{q+1}^{-\frac{14}{16}},~\bar{\bar{r}}=\lambda_{q+1}^{-\frac{5}{16}}$. A similar calculation also yields that
$$\|\dlq\|_{L^2}\lesssim \ell^{-6}_{q}\lambda^{2\alpha}_{q+1}( \mu^{-1}r^{-1/2}\bar{\bar{r}}^{-3/2}+\sigma^{-1})
\ll\lambda^{-50\alpha}_{q+1}\delta_{q+2} $$
and
$$\|(\wlq,\dlq)\|_{H^3}\lesssim\lambda^{5-50\alpha}_{q+1}\delta_{q+2}.$$
This completes the proof of \eqref{estimation-wlq}.

Finally, we turn to prove \eqref{estimation-wsq}. Recalling that $$\beta=\frac{\bar{\beta}}{b^4},~b=2^{16\lceil\bar{\beta}^{-1/2}\rceil} ~~\text{and}~~ \alpha\leq\min\Big\{\frac{1}{b^6},\frac{\beta}{b^3}\Big\},$$
we have
  $${\ell}^{\bar{\beta}}_{q-1}<\lambda_{q-1}^{(-1-\beta b+\beta)\bar{\beta}}<\lambda_{q-1}^{-2\beta b^3-50\alpha b^2}.$$
For any $(\vin, \bin)\in (H^{\bar{\beta}}, H^{\bar{\beta}}),~\bar{\beta}>\beta$, we have
\begin{align}\label{ws1}
    \|(\wsq,\dsq)\|_{L^2} &\leq \|(\vin* \psi_{\ell_q}-\vin* \psi_{\ell_{q-1}}* \psi_{\ell_q},~\bin* \psi_{\ell_q}-\bin* \psi_{\ell_{q-1}}* \psi_{\ell_q})\|_{L^2}\notag\\
    &\lesssim  \ell^{\bar{\beta}}_{q-1}\|(\vin, \bin)\|_{H^{\bar{\beta}}}\notag\lesssim \lambda_{q+1}^{-50\alpha}\delta_{q+2},
\end{align}
and
\begin{align}
    \|(\wsq,\dsq)\|_{\dot{H}^3} &\lesssim\|(\vin, \bin)\|_{L^2}
    \ell^{-3}_{q}\leq \lambda_{q+1}^{5-50\alpha}\delta_{q+2}.
\end{align}
Hence, we derive \eqref{estimation-wsq}. Collecting the estimates \eqref{estimation-whq}--\eqref{estimation-wsq}, one obtains \eqref{wq+L2}.
\end{proof}
\begin{rem}\label{xiao}
Proposition \ref{c:estimates-for-wpq-dpq} tells us that $\wcq,\wsq,\wtq,\wlq,\wvq,\dcq,\dsq,\dtq,\dlq $ and $\dvq$ are small such that
\begin{align}\label{xiao1}
    \|(\wcq,\wsq,\wtq,\wlq,\wvq,\dcq,\dsq,\dtq,\dlq,\dvq)\|_{L^2} \le\lambda_{q+1}^{-50\alpha}\delta_{q+2}.
\end{align}
So one can omit these terms in estimating the linear errors or the oscillation errors.
\end{rem}

\section{Estimates of the stresses associated with the MHD system}

\begin{prop}[Estimate for $\Rlinear$ and $\Mlinear$]\label{Rlinear}
\begin{align}
\|\Rlinear- \mathcal R[(\del_t\wtq+\del_t\wvq)+(\del_t\wlq-\Delta\wlq-\Delta\wvq)]\|_{L^1}  \lesssim\lambda^{-50\alpha}_{q+1}\delta_{q+2},\label{Rlinear1}\\
\|\Mlinear-\mathcal{R}_a[(\del_t\dtq+\del_t\dvq)+(\del_t\dlq-\Delta\dlq-\Delta\dvq)]\|_{L^1}  \lesssim\lambda^{-50\alpha}_{q+1}\delta_{q+2}.\label{Mlinear1}
\end{align}
\end{prop}
\begin{proof}
We deduce by Proposition \ref{a-estimate}-\ref{c:estimates-for-wpq-dpq} and Lemma \ref{holder1} that
\begin{align}\label{linear1}
&\|w_{q+1}  \otimes \vv_q+ \vv_q \otimes w_{q+1}\|_{L^1}\notag\\
\leq &\| \vv_q \otimes (\wcq+\wvq+\wlq+\wtq+\wsq)+(\wcq+\wvq+\wlq+\wtq+\wsq)  \otimes \vv_q\|_{L^1}\notag\\
&~+\|(\wpq+\whq ) \otimes \vv_q+ \vv_q \otimes (\wpq+\whq )\|_{L^1}\notag\\
\leq &\lambda_{q+1}^{-50\alpha}\delta_{q+2}+\|\vv_q\|_{L^2}\|a_{v,l,k}+a_{b,l,k}
+h^{1/2}_{b,q}\|_{L^2}\|\phi_{k,\bar{k},\bar{\bar{k}}}\|_{L^1}
\notag\\
\lesssim& \lambda_{q+1}^{-50\alpha}\delta_{q+2}
+\delta_{q+1}^{1/2}\lambda^{-\frac{33}{32}+\alpha}_{q+1}
\notag\\
\lesssim &\lambda_{q+1}^{-50\alpha}\delta_{q+2},
\end{align}
and
\begin{align}\label{linear2}
\|d_{q+1}  \otimes \bb_q+ \bb_q \otimes d_{q+1}\|_{L^1} \lesssim\lambda_{q+1}^{-50\alpha}\delta_{q+2}.
\end{align}

Since $\del_t \wsq-\Delta \wsq$=0, we obtain by Proposition \ref{guji1}-\ref{c:estimates-for-wpq-dpq} and Lemma \ref{holder2} that
\begin{align}\label{linear3}
&\|\mathcal R[( \del_t w_{q+1}-\del_t \wtq-\del_t \wvq-\del_t \wlq)-(\Delta w_{q+1}-\Delta \wvq-\Delta \wlq)]\|_{L^1}\notag\\
\leq&\|\mathcal R \del_t (\wpq+\wcq+\whq)\|_{L^1}+\|\mathcal R\Delta (\wpq+\wcq+\whq+\wtq ) \|_{L^1}\notag\\
\lesssim &  \ell^{-15}_q\mu\sigma\bar{\bar{r}}^{1/2}\lambda_{q+1}^{-1+5\alpha}\delta^{1/2}_{q+1} +(\ell^{-15}_q\lambda^{1+5\alpha}_{q+1}r\bar{\bar{r}}^{1/2}\delta^{1/2}_{q+1}
+\ell^{-15}_q\mu^{-1}r^{-1}\sigma
\lambda_{q+1}^{5\alpha}\delta_{q+1})\notag\\
\lesssim & \ell^{-15}_q\lambda_{q+1}^{-1-\frac{5}{32}+\frac{17}{16}+\frac{1}{128}+5\alpha}\delta^{1/2}_{q+1} +\ell^{-15}_q\lambda_{q+1}^{1-\frac{7}{8}-\frac{5}{32}+5\alpha}\delta_{q+1}
+\ell^{-15}_q\lambda_{q+1}^{-\frac{17}{16}+\frac{7}{8}+\frac{1}{128}+5\alpha}\delta_{q+1}\notag\\
\lesssim  &\lambda_{q+1}^{-50\alpha}\delta_{q+2}.
\end{align}
Collecting \eqref{linear1}-\eqref{linear3} yields \eqref{Rlinear1}. In the same way, we can prove \eqref{Mlinear1}. Thus, we complete the proof of Proposition \ref{Rlinear}.
\end{proof}

\begin{prop}[Estimate for $\Mosc$]\label{Mosclation}
\[ \|\Mosc+\mathcal{R}_a[(\del_t\dtq+\del_t\dvq)+(\del_t\dlq-\Delta\dlq-\Delta\dvq)]\|_{L^1}  \lesssim \lambda^{-50\alpha}_{q+1}\delta_{q+2}. \]
\end{prop}
\begin{proof}
Firstly, a direct computation shows that
\begin{equation}
\begin{aligned}\label{mose1}
&\div[\Mosc+ \mathcal{R}_a[(\del_t\dtq+\del_t\dvq)+(\del_t\dlq-\Delta\dlq-\Delta\dvq)]]\\
=&\mathbb{P}_H\div[\wpqb\otimes\dpq-\dpq\otimes\wpqb + M_{low,0}+\MMM_q]+ (\del_t\dtq+\del_t\dvq)+(\del_t\dlq-\Delta\dlq-\Delta\dvq)\\
=&\mathbb{P}_H\div[\wpqb\otimes\dpq-\dpq\otimes\wpqb+\MMM_q]+\div M_{low,0}+ (\del_t\dtq+\del_t\dvq)+(\del_t\dlq-\Delta\dlq-\Delta\dvq),
\end{aligned}
\end{equation}
where $M_{low,0}$ is anti-symmetric such that
\begin{align}
M_{low,0}:=& (w_{q+1}-\wpq-\whq)\otimes (d_{q+1}-\dpq-\dhq)+(\wpq+\whq)\otimes (d_{q+1}-\dpq-\dhq)\notag\\
&+(w_{q+1}-\wpq-\whq)\otimes (\dpq+\dhq)
-(d_{q+1}-\dpq-\dhq)\otimes (w_{q+1}-\wpq-\whq)\notag\\
&-(d_{q+1}-\dpq-\dhq)\otimes (\wpq+\whq)-(\dpq+\dhq)\otimes (w_{q+1}-\wpq-\whq).
\end{align}

With aid of \eqref{suppkk'}, we have $\dpq\otimes\wpqu=\wpqu\otimes\dpq=\whq\otimes\dpq=\dhq\otimes\wpq=0$.  We can easily deduce by Proposition \ref{c:estimates-for-wpq-dpq} that
\begin{align}\label{mlow0}
\|M_{low,0}\|_{L^{1}}\lesssim \lambda^{-50\alpha}_{q+1} \delta_{q+2}.
\end{align}

Secondly, recalling the definitions of $\dvq,\dlq$ in \eqref{e:defn-dvq} and \eqref{e:defn-dlq}, we obtain that
\begin{align}\label{mose2}
&\del_t\dlq-\Delta\dlq-\Delta\dvq\notag\\
=~&\Big\{(\mu\sigma)^{-1}\mathbb{P}_{H}\mathbb{P}_{>0}
\sum_{l;k\in\Lambda_{b}}  (\del_t -\Delta )(a^2_{b,l,k})
\int_{0}^{t}e^{(t-\tau)\Delta}\Delta\div\big( (\del^{-1}_t\mathbb{P}_{>0} \phi^2_{{\bar{k}}})\phi^2_{k}\phi^2_{\bar{\bar{k}}}
\bar{\bar{k}}\otimes{\bar{k}}\big)\dd\tau
\notag\\
+&\mathbb{P}_{H}\mathbb{P}_{>0}\sum_{l;k\in\Lambda_{b}}  (\del_t -\Delta )(a^2_{b,l,k})
\int_{0}^{t}e^{(t-\tau)\Delta}\div\big(\phi^2_{k}\phi^2_{\bar{\bar{k}}}
\bar{\bar{k}}\otimes{\bar{k}}\big)\dd\tau\notag\\
-&(\mu\sigma)^{-1}\mathbb{P}_{H}\mathbb{P}_{>0}\sum_{l;k\in\Lambda_{b}}
 \Delta (a^2_{b,l,k})
\div\big((\del^{-1}_t\mathbb{P}_{>0}\phi^2_{{\bar{k}}})\phi^2_{k}\phi^2_{\bar{\bar{k}}}
\bar{\bar{k}}\otimes\bar{k} \big)\notag\\
-&
(\mu\sigma)^{-1}\mathbb{P}_{H}\mathbb{P}_{>0}
\sum_{l;k\in\Lambda_{b}}\sum_{i=1}^{3}  \partial_{x_i}(a^2_{b,l,k})
\partial_{x_i}\int_{0}^{t}e^{(t-\tau)\Delta}\Delta\div\big( (\del^{-1}_t\mathbb{P}_{>0} \phi^2_{{\bar{k}}} )\phi^2_{k}\phi^2_{{\bar{\bar{k}}}}
\bar{\bar{k}}\otimes{\bar{k}}\big)\dd\tau
\notag\\
-&\mathbb{P}_{H}\mathbb{P}_{>0}\sum_{l;k\in\Lambda_{b}} \sum_{i=1}^{3}  \partial_{x_i}(a^2_{b,l,k})
\partial_{x_i}\int_{0}^{t}e^{(t-\tau)\Delta}\div\big(\phi^2_{k}\phi^2_{\bar{\bar{k}}}
\bar{\bar{k}}\otimes{\bar{k}}\big)d\tau \notag\\
-&(\mu\sigma)^{-1}\mathbb{P}_{H}\mathbb{P}_{>0}\sum_{l;k\in\Lambda_{b}}
 \sum_{i=1}^{3}  \partial_{x_i} (a^2_{b,l,k})
\partial_{x_i}\div\big((\del^{-1}_t\mathbb{P}_{>0}\phi^2_{{\bar{k}}})\phi^2_{k}\phi^2_{\bar{\bar{k}}}
\bar{\bar{k}}\otimes\bar{k} \big)\Big\}\notag\\
+&\Big\{\mathbb{P}_{H}\mathbb{P}_{>0}\sum_{l;k\in\Lambda_{b}}  a^2_{b,l,k}
\div\big(\phi^2_{k}\phi^2_{ {\bar{\bar{k}}}}
\bar{\bar{k}}\otimes{ {\bar{k}}}\big)\Big\}\notag\\
=~&\div M_{low,1}+\div M_{high,1}.
\end{align}
On one hand hand, it's easily to deduce that
\begin{align}\label{mose2-}
\Big\|&\mathcal{R}_a\Big[(\mu\sigma)^{-1}\mathbb{P}_{H}\mathbb{P}_{>0}
\sum_{l;k\in\Lambda_{b}}  (\del_t -\Delta )(a^2_{b,l,k})
\int_{0}^{t}e^{(t-\tau)\Delta}\Delta\div\big( (\del^{-1}_t\mathbb{P}_{>0} \phi^2_{{\bar{k}}})\phi^2_{k}\phi^2_{\bar{\bar{k}}}
\bar{\bar{k}}\otimes{\bar{k}}\big)\dd\tau
\notag\\
+&\mathbb{P}_{H}\mathbb{P}_{>0}\sum_{l;k\in\Lambda_{b}}  (\del_t -\Delta )(a^2_{b,l,k})
\int_{0}^{t}e^{(t-\tau)\Delta}\div\big(\phi^2_{k}\phi^2_{\bar{\bar{k}}}
\bar{\bar{k}}\otimes{\bar{k}}\big)\dd\tau\notag\\
-&(\mu\sigma)^{-1}\mathbb{P}_{H}\mathbb{P}_{>0}\sum_{l;k\in\Lambda_{b}}
 \Delta (a^2_{b,l,k})
\div\big((\del^{-1}_t\mathbb{P}_{>0}\phi^2_{{\bar{k}}})\phi^2_{k}\phi^2_{\bar{\bar{k}}}
\bar{\bar{k}}\otimes\bar{k} \big)\Big]\Big\|_{L^1}\lesssim \lambda^{-50\alpha}_{q+1} \delta_{q+2}.
\end{align}
On the other hand, using the Leibniz rule
$$\partial_{x_i}f\cdot \partial_{x_i}g=\partial_{x_i}(\partial_{x_i}f\cdot  g)-\partial^2_{x_i}f\cdot g,$$
 one can easily deduce that
 $$\|\mathcal{R}_a\mathbb{P}_{H}\mathbb{P}_{>0}(\partial_{x_i}f\cdot \partial_{x_i}g)\|_{L^1}\lesssim \|(\partial_{x_i}f\cdot  g)\|_{L^{1+\alpha}}+ \|(\partial^2_{x_i}f\cdot  g)\|_{L^{1+\alpha}}.$$
 We deduce that
 \begin{align}\label{mose2--}
\Big\|&\mathcal{R}_a\Big[(\mu\sigma)^{-1}\mathbb{P}_{H}\mathbb{P}_{>0}
\sum_{l;k\in\Lambda_{b}}\sum_{i=1}^{3}  \partial_{x_i}(a^2_{b,l,k})
\partial_{x_i}\int_{0}^{t}e^{(t-\tau)\Delta}\Delta\div\big( (\del^{-1}_t\mathbb{P}_{>0} \phi^2_{{\bar{k}}} )\phi^2_{k}\phi^2_{{\bar{\bar{k}}}}
\bar{\bar{k}}\otimes{\bar{k}}\big)\dd\tau
\notag\\
+&\mathbb{P}_{H}\mathbb{P}_{>0}\sum_{l;k\in\Lambda_{b}} \sum_{i=1}^{3}  \partial_{x_i}(a^2_{b,l,k})
\partial_{x_i}\int_{0}^{t}e^{(t-\tau)\Delta}\div\big(\phi^2_{k}\phi^2_{\bar{\bar{k}}}
\bar{\bar{k}}\otimes{\bar{k}}\big)d\tau \notag\\
+&(\mu\sigma)^{-1}\mathbb{P}_{H}\mathbb{P}_{>0}\sum_{l;k\in\Lambda_{b}}
 \sum_{i=1}^{3}  \partial_{x_i} (a^2_{b,l,k})
\partial_{x_i}\div\big((\del^{-1}_t\mathbb{P}_{>0}\phi^2_{{\bar{k}}})\phi^2_{k}\phi^2_{\bar{\bar{k}}}
\bar{\bar{k}}\otimes\bar{k} \big)\Big]\Big\|_{L^1}\lesssim \lambda^{-50\alpha}_{q+1} \delta_{q+2}.
\end{align}
Combining \eqref{mose2} with \eqref{mose2-}-\eqref{mose2--}, we obtain
\begin{align}\label{mlow1buchong}
\|M_{low,1}\|_{L^{1}}\lesssim \lambda^{-50\alpha}_{q+1} \delta_{q+2}.
\end{align}

Thirdly, using the definitions of $\eta_l,\Lambda_{v},\Lambda_{b}$ and $\phi_{k,\bar{k},\bar{\bar{k}}}$, we have
\begin{align}
\sum_{l;k\in \Lambda_{v}} a_{v,l,k} \psi_k\phi_{k,\bar{k},\bar{\bar{k}}}
\cdot\sum_{l;k\in \Lambda_{b}} a_{b,l,k} \psi_k\phi_{k,\bar{k},\bar{\bar{k}}}
&=0,
\end{align}
and
\begin{align}
\sum_{l;k\in \Lambda_{b}} a_{b,l,k} \psi_k\phi_{k,\bar{k},\bar{\bar{k}}}
\cdot\sum_{l;k\in \Lambda_{b}} a_{b,l,k} \psi_k\phi_{k,\bar{k},\bar{\bar{k}}}
&=\sum_{l;k\in \Lambda_{b}} a^2_{b,l,k} \psi^2_k\phi^2_{k,\bar{k},\bar{\bar{k}}}.
\end{align}
By $\div=\div\mathbb{P}_{>0}=\mathbb{P}_{>0}\div$, we show that
\begin{align}\label{Mose2}
&\mathbb{P}_H \div\Big[ \wpqb\otimes\dpq-\dpq\otimes\wpqb + \MMM_q\Big]+(\del_t\dtq+\del_t\dvq+\div M_{high,1} )\notag\\
=&\mathbb{P}_H \div\Big[ \sum_{l;k\in \Lambda_{b}} a^2_{b,l,k} \psi^2_k(\phi^2_k\phi^2_{\bar{k}}\phi^2_{\bar{\bar{k}}})
(\bar{k}\otimes\bar{\bar{k}}-\bar{\bar{k}}\otimes\bar{k})+\MMM_q\Big]
+(\del_t\dtq+\del_t\dvq+\div M_{high,1} )\notag\\
 =&\mathbb{P}_H \div\Big[ \sum_{l;k\in \Lambda_{b}} a^2_{b,l,k}(\bar{k}\otimes\bar{\bar{k}}-\bar{\bar{k}}\otimes\bar{k})+ \MMM_q\Big]+ \mathbb{P}_H \div\Big[\sum_{l;k\in \Lambda_{b}} a^2_{b,l,k}
\mathbb{P}_{>0}(\psi^2_k)(\phi^2_k\phi^2_{\bar{k}}\phi^2_{\bar{\bar{k}}})
(\bar{k}\otimes\bar{\bar{k}}-\bar{\bar{k}}\otimes\bar{k})\Big]\notag\\
&+ \mathbb{P}_H \div\Big[\sum_{l;k\in \Lambda_{b}} a^2_{b,l,k}
\mathbb{P}_{>0}(\phi^2_k\phi^2_{\bar{k}}\phi^2_{\bar{\bar{k}}})
(\bar{k}\otimes\bar{\bar{k}}-\bar{\bar{k}}\otimes\bar{k})\Big]+(\del_t\dtq+\del_t\dvq+\div M_{high,1} )\notag\\
 =&\Big\{0+ \mathbb{P}_H \mathbb{P}_{>0}\sum_{l;k\in \Lambda_{b}}\mathbb{P}_{>0}(\psi^2_k)
\div\Big[ a^2_{b,l,k}
(\phi^2_k\phi^2_{\bar{k}}\phi^2_{\bar{\bar{k}}})
(\bar{k}\otimes\bar{\bar{k}}-\bar{\bar{k}}\otimes\bar{k})\Big]\notag\\
& +\mathbb{P}_H\mathbb{P}_{>0}\sum_{l;k\in \Lambda_{b}}
\mathbb{P}_{>0}(\phi^2_k\phi^2_{\bar{k}}\phi^2_{\bar{\bar{k}}})\div[ a^2_{b,l,k}
(\bar{k}\otimes\bar{\bar{k}}-\bar{\bar{k}}\otimes\bar{k})]\Big\}\notag\\
& +\Big\{\mathbb{P}_H\mathbb{P}_{>0}\sum_{l;k\in \Lambda_{b}}a^2_{b,l,k}
\div[(\phi^2_k\phi^2_{\bar{k}}\phi^2_{\bar{\bar{k}}})
(\bar{k}\otimes\bar{\bar{k}})]+\del_t\dtq\Big\}\notag\\
& +\Big\{\mathbb{P}_H\mathbb{P}_{>0}\sum_{l;k\in \Lambda_{b}}-a^2_{b,l,k}\div
[(\phi^2_k\phi^2_{\bar{k}}\phi^2_{\bar{\bar{k}}})(\bar{\bar{k}}\otimes\bar{k})]
  +\del_t\dvq+\div M_{high,1}\Big\}\notag\\
 :=&\div M_{low,2}+\div M_{low,3}+\div M_{low,4},
\end{align}
where we have used Lemma \ref{first L} in the third equality. Combining \eqref{mose1},~\eqref{mose2} and \eqref{Mose2} shows
\begin{align}\label{Mose1}
&\div[\Mosc+\mathcal{R}_a[(\del_t\dtq+\del_t\dvq)+(\del_t\dlq-\Delta\dlq-\Delta\dvq)]]\notag\\
=&\div (M_{low,0}+ M_{low,1}+M_{low,2}+ M_{low,3}+M_{low,4}).
\end{align}


We begin to estimate $M_{low,2}$. Using Lemma \ref{holder2} and Proposition \ref{a-estimate}, we deduce that
\begin{align}\label{mlow1}
\|M_{low,2}\|_{L^1}\lesssim \ell^{-15}_qr^{-1}\lambda^{-1+\sigma}_{q+1}
\ll\lambda^{-50\alpha}_{q+1}\delta_{q+2}.
\end{align}
Next, we consider $M_{low,3}$. Recalling that
$$\del_t\dtq=-\tfrac{1}{{\mu}}\mathbb{P}_{H}\mathbb{P}_{>0}\sum_{l;k\in\Lambda_{b}}
\Big\{\del_t(a^2_{b,l,k})
(\phi^2_k\phi^2_{\bar{k}}\phi^2_{\bar{\bar{k}}})
\bar{\bar{k}}+a^2_{b,l,k}\del_t\phi^2_{\bar{k}}\cdot
 \phi^2_k\phi^2_{\bar{\bar{k}}}
 \bar{\bar{k}}\Big\},$$
by $$\del_t[\phi^2_{\bar{k}}(x,t)]\cdot
 \phi^2_k\phi^2_{\bar{\bar{k}}}
 \bar{\bar{k}}= \mu\div[(\phi^2_k\phi^2_{\bar{k}}\phi^2_{\bar{\bar{k}}})(\bar{k}\otimes\bar{\bar{k}})],$$ we have
\begin{align}
M_{low,3}
&=\mathcal{R}_a \mathbb{P}_{H}\mathbb{P}_{>0}\sum_{l;k\in \Lambda_{b}} a^2_{b,l,k}
\div[(\phi^2_k\phi^2_{\bar{k}}\phi^2_{\bar{\bar{k}}})(\bar{k}\otimes\bar{\bar{k}})]+\mathcal{R}_a\del_t\dtq \notag\\
&=-\mathcal{R}_a\mathbb{P}_{H}\mathbb{P}_{>0}\sum_{l;k\in\Lambda_{b}}
{{\mu}}^{-1}\del_t(a^2_{b,l,k})
(\phi^2_k\phi^2_{\bar{k}}\phi^2_{\bar{\bar{k}}})
\bar{\bar{k}}.
\end{align}
So we can easily deduce that
\begin{align}\label{mlow2}
\|M_{low,3}\|_{L^1}\lesssim \mu^{-1}\|\del_t(a^2_{b,l,k})\|_{L^{\infty}}
\|\phi^2_k\phi^2_{\bar{k}}\phi^2_{\bar{\bar{k}}}\|_{L^{1+\alpha}}\ll\lambda^{-50\alpha}_{q+1}\delta_{q+2}.
\end{align}
Finally, we need to estimate $M_{low,4}$. Using the definition of $M_{high,1}$ in \eqref{mose2} and $\dvq$ in \eqref{e:defn-dvq}, we deduce that
\begin{align}
 M_{low,4}=&\mathcal{R}_a\mathbb{P}_H\mathbb{P}_{>0}\sum_{l;k\in \Lambda_{b}}-a^2_{b,l,k}\div
[(\mathbb{P}_{>0}(\phi^2_{\bar{k}})+1)\cdot\phi^2_k\phi^2_{\bar{\bar{k}}}(\bar{\bar{k}}\otimes\bar{k})]
  +\del_t\dvq+M_{high,1} \notag\\
  =&\mathcal{R}_a\mathbb{P}_H\mathbb{P}_{>0}\sum_{l;k\in \Lambda_{b}}-a^2_{b,l,k}\div
[\mathbb{P}_{>0}(\phi^2_{\bar{k}})\phi^2_k\phi^2_{\bar{\bar{k}}}(\bar{\bar{k}}\otimes\bar{k})]
  +\del_t\dvq\notag\\
  =&(\mu\sigma)^{-1}\mathcal{R}_a\mathbb{P}_H\mathbb{P}_{>0}\sum_{l;k\in \Lambda_{b}}\del_t(a^2_{b,l,k})\div
[(\del^{-1}_t\mathbb{P}_{>0}\phi^2_{\bar{k}})\phi^2_k\phi^2_{\bar{\bar{k}}}
(\bar{\bar{k}}\otimes\bar{k})]\notag\\
  =&(\mu\sigma)^{-1}\mathcal{R}_a\mathbb{P}_H\mathbb{P}_{>0}\sum_{l;k\in \Lambda_{b}}(\del_ta^2_{b,l,k})
(\del^{-1}_t\mathbb{P}_{>0}\phi^2_{\bar{k}})\phi^2_k\div[\phi^2_{\bar{\bar{k}}}(
\bar{\bar{k}}\otimes\bar{k})].
\end{align}
 One deduces by Proposition \ref{a-estimate} that
\begin{align}\label{mlow3}
\|M_{low,4}\|_{L^1}&\leq {(\mu\sigma)}^{-1}\|\del_t(a^2_{b,l,k})\|_{L^{\infty}}\| (\del^{-1}_t\mathbb{P}_{>0}\phi^2_{\bar{k}})\phi^2_k\div(\phi^2_{\bar{\bar{k}}})
\bar{\bar{k}}\otimes\bar{k}\|_{L^{1+\alpha}} \notag\\
&\lesssim  \ell^{-15}_q\lambda^{2\alpha}_{q+1}\mu^{-1} \bar{\bar{r}}^{-1}\notag\\
&\ll\lambda^{-50\alpha}_{q+1}\delta_{q+2}.
\end{align}
Putting \eqref{mlow0},\eqref{mlow1buchong},\eqref{mlow1},\eqref{mlow2} and \eqref{mlow3} into \eqref{Mose1}, we obtain
 $$\|\Mosc+\mathcal{R}_a[(\del_t\dtq+\del_t\dvq)+(\del_t\dlq-\Delta\dlq-\Delta\dvq)]\|_{L^1} \ll \lambda^{-50\alpha}_{q+1}\delta_{q+2}.$$
 This completes the proof of Propositon \ref{Mosclation}.
 \end{proof}

\begin{prop}[Estimate for $\Rosc$]\label{Rosclation}
\[ \|\Rosc+\mathcal{R}[(\del_t\wtq+\del_t\wvq)+(\del_t\wlq-\Delta\wlq-\Delta\wvq) ]\|_{L^1} \lesssim \lambda^{-50\alpha}_{q+1}\delta_{q+2}, \]
where $P_v$ is defined in \eqref{Pu}.
\end{prop}
\begin{proof}
Firstly, since $\whq=\dhq$ and $\wpqu\otimes\wpqb=\whq\otimes\wpq=\dhq\otimes\dpq=0$, we have
\begin{align}\label{rosc1}
w_{q+1}\otimes w_{q+1}-d_{q+1}\otimes d_{q+1}=\wpq\otimes\wpq-\dpq\otimes\dpq +R_{low,0},
\end{align}
where
\begin{align}
R_{low,0}:=&(w_{q+1}-\wpq-\whq)\otimes (w_{q+1}-\wpq-\whq)+(\wpq+\whq)\otimes (w_{q+1}-\wpq-\whq)\notag\\
&+(w_{q+1}-\wpq-\whq)\otimes (\wpq+\whq)-(d_{q+1}-\dpq-\dhq)\otimes (d_{q+1}-\dpq-\dhq)\notag\\
&-(\dpq+\dhq)\otimes (d_{q+1}-\dpq-\dhq)-(d_{q+1}-\dpq-\dhq)\otimes (\dpq+\dhq) .
\end{align}
We deduce by  Proposition \ref{c:estimates-for-wpq-dpq} that
\begin{align}\label{rlow0}
\|\mathcal{R}\div R_{low,0}\|_{L^{1}}\lesssim \lambda^{-50\alpha}_{q+1} \delta_{q+2}.
\end{align}

Secondly, straightforward calculations show
\begin{align}\label{R-ose2}
&\del_t\wlq-\Delta\wlq-\Delta\wvq\notag\\
=&\Big\{(\mu\sigma)^{-1}\mathbb{P}_{H}\mathbb{P}_{>0}
\sum_{l;k\in\Lambda_{b}}  (\del_t -\Delta )(a^2_{b,l,k})
\int_{0}^{t}e^{(t-\tau)\Delta}\Delta\div\big( (\del^{-1}_t\mathbb{P}_{>0} \phi^2_{{\bar{k}}} )\phi^2_{k}\phi^2_{\bar{\bar{k}}}
\bar{\bar{k}}\otimes\bar{\bar{k}}\big)\dd\tau
\notag\\
&+\mathbb{P}_{H}\mathbb{P}_{>0}\sum_{l;k\in\Lambda_{b}}  (\del_t -\Delta )(a^2_{b,l,k})
\int_{0}^{t}e^{(t-\tau)\Delta}\div\big(\phi^2_{k}\phi^2_{\bar{\bar{k}}}
\bar{\bar{k}}\otimes\bar{\bar{k}}\big)d\tau\notag\\
-&(\mu\sigma)^{-1}\mathbb{P}_{H}\mathbb{P}_{>0}\sum_{l;k\in\Lambda_{b}}
 \Delta (a^2_{b,l,k})
\div\big((\del^{-1}_t\mathbb{P}_{>0}\phi^2_{{\bar{k}}})\phi^2_{k}\phi^2_{\bar{\bar{k}}}
\bar{\bar{k}}\otimes\bar{\bar{k}} \big)\notag\\
&-
(\mu\sigma)^{-1}\mathbb{P}_{H}\mathbb{P}_{>0}
\sum_{l;k\in\Lambda_{b}}\sum_{i=1}^{3}  \partial_{x_i}(a^2_{b,l,k})
\partial_{x_i}\int_{0}^{t}e^{(t-\tau)\Delta}\Delta\div\big( (\del^{-1}_t\mathbb{P}_{>0} \phi^2_{{\bar{k}}} )\phi^2_{k}\phi^2_{\bar{\bar{k}}}
\bar{\bar{k}}\otimes\bar{\bar{k}}\big)\dd\tau
\notag\\
&-\mathbb{P}_{H}\mathbb{P}_{>0}\sum_{l;k\in\Lambda_{b}} \sum_{i=1}^{3}  \partial_{x_i}(a^2_{b,l,k})
\partial_{x_i}\int_{0}^{t}e^{(t-\tau)\Delta}\div\big(\phi^2_{k}\phi^2_{\bar{\bar{k}}}
\bar{\bar{k}}\otimes\bar{\bar{k}}\big)\dd\tau\notag\\
-&(\mu\sigma)^{-1}\mathbb{P}_{H}\mathbb{P}_{>0}\sum_{l;k\in\Lambda_{b}}
 \sum_{i=1}^{3}  \partial_{x_i} (a^2_{b,l,k})
\partial_{x_i}\div\big((\del^{-1}_t\mathbb{P}_{>0}\phi^2_{{\bar{k}}})\phi^2_{k}\phi^2_{\bar{\bar{k}}}
\bar{\bar{k}}\otimes\bar{\bar{k}} \big)\Big\}\notag\\
&+\Big\{\mathbb{P}_{H}\mathbb{P}_{>0}\sum_{l;k\in\Lambda_{b}}  a^2_{b,l,k}
\div\big(\phi^2_{k}\phi^2_{ {\bar{\bar{k}}}}
\bar{\bar{k}}\otimes\bar{\bar{k}}\big)\Big\}\notag\\
=&\div R_{low,1}+\div R_{high,1}.
\end{align}
On one hand, it's easily to deduce that
\begin{align}\label{Rose2-}
\Big\|&\mathcal{R}\Big[(\mu\sigma)^{-1}\mathbb{P}_{H}\mathbb{P}_{>0}
\sum_{l;k\in\Lambda_{b}}  (\del_t -\Delta )(a^2_{b,l,k})
\int_{0}^{t}e^{(t-\tau)\Delta}\Delta\div\big( (\del^{-1}_t\mathbb{P}_{>0} \phi^2_{{\bar{k}}})\phi^2_{k}\phi^2_{\bar{\bar{k}}}
\bar{\bar{k}}\otimes{\bar{\bar{k}}}\big)\dd\tau
\notag\\
+&\mathbb{P}_{H}\mathbb{P}_{>0}\sum_{l;k\in\Lambda_{b}}  (\del_t -\Delta )(a^2_{b,l,k})
\int_{0}^{t}e^{(t-\tau)\Delta}\div\big(\phi^2_{k}\phi^2_{\bar{\bar{k}}}
\bar{\bar{k}}\otimes{\bar{\bar{k}}}\big)\dd\tau\notag\\
-&(\mu\sigma)^{-1}\mathbb{P}_{H}\mathbb{P}_{>0}\sum_{l;k\in\Lambda_{b}}
 \Delta (a^2_{b,l,k})
\div\big((\del^{-1}_t\mathbb{P}_{>0}\phi^2_{{\bar{k}}})\phi^2_{k}\phi^2_{\bar{\bar{k}}}
\bar{\bar{k}}\otimes\bar{\bar{k}} \big)\Big]\Big\|_{L^1}\lesssim \lambda^{-50\alpha}_{q+1} \delta_{q+2}.
\end{align}
On the other hand, using the Leibniz rule
$$\partial_{x_i}f\cdot \partial_{x_i}g=\partial_{x_i}(\partial_{x_i}f\cdot  g)-\partial^2_{x_i}f\cdot g,$$
 one can easily deduce that
 $$\|\mathcal{R}\mathbb{P}_{>0}(\partial_{x_i}f\cdot \partial_{x_i}g)\|_{L^1}\lesssim \|(\partial_{x_i}f\cdot  g)\|_{L^{1+\alpha}}+ \|(\partial^2_{x_i}f\cdot  g)\|_{L^{1+\alpha}}.$$
 We deduce that
 \begin{align}\label{Rose2--}
\Big\|&\mathcal{R}\Big[(\mu\sigma)^{-1}\mathbb{P}_{H}\mathbb{P}_{>0}
\sum_{l;k\in\Lambda_{b}}\sum_{i=1}^{3}  \partial_{x_i}(a^2_{b,l,k})
\partial_{x_i}\int_{0}^{t}e^{(t-\tau)\Delta}\Delta\div\big( (\del^{-1}_t\mathbb{P}_{>0} \phi^2_{{\bar{k}}} )\phi^2_{k}\phi^2_{{\bar{\bar{k}}}}
\bar{\bar{\bar{k}}}\otimes{\bar{k}}\big)\dd\tau
\notag\\
+&\mathbb{P}_{H}\mathbb{P}_{>0}\sum_{l;k\in\Lambda_{b}} \sum_{i=1}^{3}  \partial_{x_i}(a^2_{b,l,k})
\partial_{x_i}\int_{0}^{t}e^{(t-\tau)\Delta}\div\big(\phi^2_{k}\phi^2_{\bar{\bar{k}}}
\bar{\bar{\bar{k}}}\otimes{\bar{k}}\big)d\tau \notag\\
+&(\mu\sigma)^{-1}\mathbb{P}_{H}\mathbb{P}_{>0}\sum_{l;k\in\Lambda_{b}}
 \sum_{i=1}^{3}  \partial_{x_i} (a^2_{b,l,k})
\partial_{x_i}\div\big((\del^{-1}_t\mathbb{P}_{>0}\phi^2_{{\bar{k}}})\phi^2_{k}\phi^2_{\bar{\bar{k}}}
\bar{\bar{k}}\otimes\bar{\bar{k}} \big)\Big]\Big\|_{L^1}\lesssim \lambda^{-50\alpha}_{q+1} \delta_{q+2}.
\end{align}
Combining \eqref{R-ose2} with \eqref{Rose2-}-\eqref{Rose2--}, we obtain
\begin{align}\label{rlow1buchong}
\|R_{low,1}\|_{L^{1}}\lesssim \lambda^{-50\alpha}_{q+1} \delta_{q+2}.
\end{align}

Thirdly, we deduce that
\begin{align}\label{Pu}
& \div( \wpq\otimes\wpq-\dpq\otimes\dpq + \RRR_q )+(\del_t\wtq+\del_t\wtq+\div R_{high,1})\notag\\
=&\div\Big[ \sum_{l;k \in \Lambda_{v}} a^2_{v,l,k} \psi^2_k
(\phi^2_k\phi^2_{\bar{k}}\phi^2_{\bar{\bar{k}}})(\bar{k}\otimes\bar{k})
+\RRR_q\Big]+(\del_t\wtq+\del_t\wtq+\div R_{high,1})\notag\\
& +\div\Big[\sum_{l;k \in \Lambda_{b}}a^2_{b,l,k} \psi^2_k
 (\phi^2_k\phi^2_{\bar{k}}\phi^2_{\bar{\bar{k}}})(\bar{k}\otimes\bar{k})
-\sum_{l;k \in \Lambda_{b}} a^2_{b,l,k} \psi^2_k
 (\phi^2_k\phi^2_{\bar{k}}\phi^2_{\bar{\bar{k}}})(\bar{\bar{k}}\otimes\bar{\bar{k}})
\Big]\notag\\
=&\div\Big[\sum_{l;k \in \Lambda_{v}}a^2_{v,l,k} \mathbb{P}_{>0}(\psi^2_k)
(\phi^2_k\phi^2_{\bar{k}}\phi^2_{\bar{\bar{k}}})(\bar{k}\otimes\bar{k})
+a^2_{v,l,k} \mathbb{P}_{>0}
(\phi^2_k\phi^2_{\bar{k}}\phi^2_{\bar{\bar{k}}})(\bar{k}\otimes\bar{k})\Big]+\nabla\sum_{l} \eta^2_{l}\chi_{v}\rho_{v,q}\notag\\
&+\div\Big[\sum_{l;k \in \Lambda_{b}}a^2_{b,l,k} \mathbb{P}_{>0}(\psi^2_k)
(\phi^2_k\phi^2_{\bar{k}}\phi^2_{\bar{\bar{k}}})(\bar{k}\otimes\bar{k})
+a^2_{b,l,k} \mathbb{P}_{>0}
(\phi^2_k\phi^2_{\bar{k}}\phi^2_{\bar{\bar{k}}})(\bar{k}\otimes\bar{k})\Big]
+\del_t\wtq\notag\\
&-\div\Big[\sum_{l;k \in \Lambda_{b}}a^2_{b,l,k} \mathbb{P}_{>0}(\psi^2_k)
(\phi^2_k\phi^2_{\bar{k}}\phi^2_{\bar{\bar{k}}})(\bar{\bar{k}}\otimes\bar{\bar{k}})
+a^2_{b,l,k} \mathbb{P}_{>0}
(\phi^2_k\phi^2_{\bar{k}}\phi^2_{\bar{\bar{k}}})(\bar{\bar{k}}\otimes\bar{\bar{k}})\Big]
+\del_t\wvq+\div R_{high,1}\notag\\
=&\div\Big[\sum_{l;k \in \Lambda_{v}}a^2_{v,l,k} \mathbb{P}_{>0}(\psi^2_k)
(\phi^2_k\phi^2_{\bar{k}}\phi^2_{\bar{\bar{k}}})(\bar{k}\otimes\bar{k})
+\sum_{l;k \in \Lambda_{b}}a^2_{b,l,k} \mathbb{P}_{>0}(\psi^2_k)
(\phi^2_k\phi^2_{\bar{k}}\phi^2_{\bar{\bar{k}}})
(\bar{k}\otimes\bar{k}-\bar{\bar{k}}\otimes\bar{\bar{k}})\Big]\notag\\
&
+\div\Big[\sum_{l;k \in \Lambda_{v}}a^2_{v,l,k} \mathbb{P}_{>0}
(\phi^2_k\phi^2_{\bar{k}}\phi^2_{\bar{\bar{k}}})(\bar{k}\otimes\bar{k})
+\sum_{l;k \in \Lambda_{b}}a^2_{b,l,k} \mathbb{P}_{>0}
(\phi_k\phi_{\bar{k}}\phi_{\bar{\bar{k}}})(\bar{k}\otimes\bar{k})+\mathcal{R}\del_t\wtq\Big]\notag\\
&+\div\Big[\sum_{l;k \in \Lambda_{b}}-a^2_{b,l,k} \mathbb{P}_{>0}
(\phi^2_k\phi^2_{\bar{k}}\phi^2_{\bar{\bar{k}}})(\bar{\bar{k}}\otimes\bar{\bar{k}})
+\mathcal{R}\del_t\wvq+R_{high,1}\Big]+\nabla\sum_{l} \eta^2_{l}\chi_{v}\rho_{v,q}
\notag\\
=&\Big\{\mathbb{P}_{>0}\Big[\sum_{l;k \in \Lambda_{v}}\mathbb{P}_{>0}(\psi^2_k)\div [a^2_{v,l,k}
(\phi^2_k\phi^2_{\bar{k}}\phi^2_{\bar{\bar{k}}})(\bar{k}\otimes\bar{k})]
+\sum_{l;k \in \Lambda_{b}}\mathbb{P}_{>0}(\psi^2_k)\div [a^2_{b,l,k}
(\phi_k\phi_{\bar{k}}\phi_{\bar{\bar{k}}})
(\bar{k}\otimes\bar{k}-\bar{\bar{k}}\otimes\bar{\bar{k}})]\Big]\notag\\
&+\mathbb{P}_{>0}\Big[\sum_{l;k \in \Lambda_{v}} \mathbb{P}_{>0}
(\phi^2_k\phi^2_{\bar{k}}\phi^2_{\bar{\bar{k}}})\div [a^2_{v,l,k}(\bar{k}\otimes\bar{k})]
+\sum_{l;k \in \Lambda_{b}}\mathbb{P}_{>0}
(\phi^2_k\phi^2_{\bar{k}}\phi^2_{\bar{\bar{k}}})\div[ a^2_{b,l,k}(\bar{k}\otimes\bar{k}-\bar{\bar{k}}\otimes\bar{\bar{k}})]\Big]\Big\}\notag\\
&
+\Big\{\mathbb{P}_{>0}\Big[\sum_{l;k \in \Lambda_{v}} a^2_{v,l,k} \div
[(\phi^2_k\phi^2_{\bar{k}}\phi^2_{\bar{\bar{k}}}) (\bar{k}\otimes\bar{k})]
+\sum_{l;k \in \Lambda_{b}}  a^2_{b,l,k}\div[
(\phi^2_k\phi^2_{\bar{k}}\phi^2_{\bar{\bar{k}}})(\bar{k}\otimes\bar{k})]+ \del_t\wtq\Big]\Big\}\notag\\
&+\Big\{\mathbb{P}_{>0}\Big[\sum_{l;k \in \Lambda_{b}}-a^2_{b,l,k}   \div[
(\phi^2_k\phi^2_{\bar{k}}\phi^2_{\bar{\bar{k}}}) (\bar{\bar{k}}\otimes\bar{\bar{k}})]+\del_t\wvq+\div R_{high,1}\Big]\Big\}+\nabla\sum_{l} \eta^2_{l}\chi_{v}\rho_{v,q}
\notag\\
:=&\div R_{low,2}+\div R_{low,3}+\div R_{low,4}+\nabla\sum_{l} \eta^2_{l}\chi_{v}\rho_{v,q}\notag\\
:=&\div R_{low,2}+\mathbb{P}_{H}\div R_{low,3}+\mathbb{P}_{H}\div R_{low,4}+\nabla(\frac{\div\div}{\Delta}R_{low,3}+\frac{\div\div}{\Delta}R_{low,4}+\sum_{l} \eta^2_{l}\chi_{v}\rho_{v,q})\notag\\
:=&\div R_{low,2}+\mathbb{P}_{H}\div R_{low,3}+\mathbb{P}_{H}\div R_{low,4}+\nabla P_v ,
\end{align}
where we use Lemma \ref{first S} with $R_v=\RRR_q-\sum_{k\in \Lambda_{v}}a_{b,l,k}(\bar{k}\otimes\bar{k}
-\bar{\bar{k}}\otimes\bar{\bar{k}})$ in the second equality. Thus, by the definition of $\Rosc$, we deduce
\begin{align}\label{RR}
\div\Big[\Rosc +\mathcal{R}[(\del_t\wtq+\del_t\wvq)+(\del_t\wlq-\Delta\wlq-\Delta\wvq) ]\Big]\notag\\
=\div\Big[R_{low,0}+  R_{low,1}+\mathcal{R}\div  (R_{low,2})+\mathcal{R}\mathbb{P}_{H}\div (R_{low,3}+ R_{low,4})\Big].
\end{align}
We need to estimate $\mathcal{R}\div R_{low,2}$, $\mathcal{R}\mathbb{P}_{H}\div R_{low,3}$ and $\mathcal{R}\mathbb{P}_{H}\div R_{low,4}$, respectively.

To begin with, we estimate $\mathcal{R}\div R_{low,2}$.  Thanks to Lemma \ref{holder2} and Proposition \ref{a-estimate},  one obtains  that
\begin{align}\label{rlow1}
\|\mathcal{R}\div R_{low,2}\|_{L^1}\lesssim \ell^{-15}_{q}\sigma r^{-1}\lambda^{-1+\alpha}_{q+1}
+\ell^{-15}_{q}\sigma^{-1}\lambda^{\alpha}_{q+1}
\ll\lambda^{-50\alpha}_{q+1}\delta_{q+2}.
\end{align}
Next, we need to $ \mathcal{R}\mathbb{P}_{H}\div R_{low,3}$. Recalling that
\begin{align}
\del_t\wtq=&-\mathbb{P}_{H}\mathbb{P}_{>0}\sum_{l;k\in\Lambda_{v}}
\tfrac{1}{\mu}\del_t(a^2_{v,l,k})
(\phi^2_k\phi^2_{\bar{k}}\phi^2_{\bar{\bar{k}}})
{\bar{k}}-\mathbb{P}_{H}\mathbb{P}_{>0}\sum_{l;k\in\Lambda_{b}}
\tfrac{1}{\mu}\del_t(a^2_{b,l,k})
(\phi^2_k\phi^2_{\bar{k}}\phi^2_{\bar{\bar{k}}})
{\bar{k}}\notag\\
&-\mathbb{P}_{H}\mathbb{P}_{>0}\sum_{l;k\in\Lambda_{v}}
a^2_{v,l,k}\div
(\phi^2_k\phi^2_{\bar{k}}\phi^2_{\bar{\bar{k}}})
\bar{k}\otimes {\bar{k}}
-\mathbb{P}_{H}\mathbb{P}_{>0}\sum_{l;k\in\Lambda_{b}}
a^2_{b,l,k}\div
(\phi^2_k\phi^2_{\bar{k}}\phi^2_{\bar{\bar{k}}})
\bar{k}\otimes {\bar{k}}.
\end{align}
Owning to $\div=\div\mathbb{P}_{>0}=\mathbb{P}_{>0}\div$, we have
\begin{align}
\mathcal{R}\mathbb{P}_{H}\div R_{low,3}=&\mathcal{R}\mathbb{P}_{H}\mathbb{P}_{>0}\sum_{l;k\in \Lambda_{v}}\big(a^2_{v,l,k}
\div
(\phi^2_k\phi^2_{\bar{k}}\phi^2_{\bar{\bar{k}}})
\bar{k}\otimes {\bar{k}}\big)\notag\\
&~+\mathcal{R}\mathbb{P}_{H}\mathbb{P}_{>0}\sum_{l;k\in \Lambda_{b}}\big(a^2_{b,l,k}
\div
(\phi^2_k\phi^2_{\bar{k}}\phi^2_{\bar{\bar{k}}})
\bar{k}\otimes {\bar{k}}\big)+\mathcal{R}\del_t\wtq\notag\\
=&-\mathcal{R}\mathbb{P}_{H}\mathbb{P}_{>0}\sum_{l;k\in\Lambda_{v}}
\tfrac{1}{\mu}\del_t(a^2_{v,l,k})
(\phi^2_k\phi^2_{\bar{k}}\phi^2_{\bar{\bar{k}}})
{\bar{k}}\notag\\
&
-\mathcal{R}\mathbb{P}_{H}\mathbb{P}_{>0}\sum_{l;k\in\Lambda_{b}}
\tfrac{1}{\mu}\del_t(a^2_{b,l,k})
(\phi^2_k\phi^2_{\bar{k}}\phi^2_{\bar{\bar{k}}})
{\bar{k}}.
\end{align}
Hence, one can easily deduce that
\begin{align}\label{rlow2}
\|\mathcal{R}\mathbb{P}_{H}\div R_{low,3}\|_{L^{1}}\lesssim  {\mu}^{-1}\|\del_t(a^2_{v,l,k},a^2_{b,l,k})\|_{L^{\infty}}
\|\phi^2_k\phi^2_{\bar{k}}\phi^2_{\bar{\bar{k}}}\|_{L^{1+\alpha}}\ll\lambda^{-5\alpha}_{q+1}\delta_{q+2}.
\end{align}
Finally, we turn to estimate $\mathcal{R}\mathbb{P}_{H}\div R_{low,4}$. Using the definition of $R_{high,1}$ in \eqref{R-ose2} and $\wvq$ in \eqref{e:defn-wvq}, we rewrite $\mathcal{R}\mathbb{P}_{H}\div R_{low,4}$ as
\begin{align}
 \mathcal{R}\mathbb{P}_{H}\div R_{low,4}
=&\mathcal{R}\mathbb{P}_H\mathbb{P}_{>0}\sum_{l;k\in \Lambda_{b}}-a^2_{b,l,k}\div
[(\mathbb{P}_{>0}\phi^2_{\bar{k}}+1)\phi^2_k\phi^2_{\bar{\bar{k}}}(\bar{\bar{k}}\otimes\bar{\bar{k}})]
  +\del_t\wvq+\div R_{high,1} \notag\\
  =&\mathcal{R} \mathbb{P}_H\mathbb{P}_{>0}\sum_{l;k\in \Lambda_{b}}-a^2_{b,l,k}\div[
\mathbb{P}_{>0}(\phi^2_{\bar{k}})\phi^2_k\phi^2_{\bar{\bar{k}}}(\bar{\bar{k}}\otimes\bar{\bar{k}})]
  +\del_t\wvq\notag\\
  =&(\mu\sigma)^{-1}\mathcal{R}\mathbb{P}_H\mathbb{P}_{>0}\sum_{l;k\in \Lambda_{b}}\del_t(a^2_{b,l,k})\div[
(\del^{-1}_t\mathbb{P}_{>0} \phi^2_{\bar{k}})\phi^2_k\phi^2_{\bar{\bar{k}}}(
\bar{\bar{k}}\otimes\bar{\bar{k}})]\notag\\
  =&(\mu\sigma)^{-1}\mathcal{R}\mathbb{P}_H\mathbb{P}_{>0}\sum_{l;k\in \Lambda_{b}}\del_t(a^2_{b,l,k})
(\del^{-1}_t\mathbb{P}_{>0}\phi^2_{\bar{k}})\phi^2_k\div[\phi^2_{\bar{\bar{k}}}(
\bar{\bar{k}}\otimes\bar{\bar{k}})].
\end{align}
 One deduces by Proposition \ref{a-estimate} that
\begin{align}\label{rlow3}
\|\mathcal{R}\mathbb{P}_{H}\div R_{low,4}\|_{L^1}&\leq ({\mu\sigma})^{-1}\|\del_t(a^2_{b,l,k})\|_{L^{\infty}}\| (\del^{-1}_t\mathbb{P}_{>0}\phi^2_{\bar{k}})\phi^2_k\div[\phi^2_{\bar{\bar{k}}}
(\bar{\bar{k}}\otimes\bar{\bar{k}})]\|_{L^{1+\alpha}} \notag\\
&\lesssim  \ell^{-15}_q\lambda^{2\alpha}_{q+1}\mu^{-1} \bar{\bar{r}}^{-1}\notag\\
&\ll\lambda^{-50\alpha}_{q+1}\delta_{q+2}.
\end{align}
Plugging \eqref{rlow0},\eqref{rlow1buchong},\eqref{rlow1},\eqref{rlow2} and \eqref{rlow3} into \eqref{RR}, we obtain
 $$\|\Rosc+\mathcal{R}[(\del_t\wtq+\del_t\wvq)+(\del_t\wlq-\Delta\wlq-\Delta\wvq) ]\|_{L^1}  \ll \lambda^{-50\alpha}_{q+1}\delta_{q+2}.$$
 This completes the proof of Proposition \ref{Rosclation}.
 \end{proof}

Collecting Proposition \ref{Rlinear}--Proposition \ref{Rosclation}, we obtain that
 \begin{align*}
    \|M_{q+1}\|_{L^1} \leq & \|M_{osc}+\mathcal{R}_a[(\del_t\dtq+\del_t\dvq)+(\del_t\dlq-\Delta\dlq-\Delta\dvq)]\|_{L^1} \\ & \|\Mlinear-\mathcal{R}_a[(\del_t\dtq+\del_t\dvq)+(\del_t\dlq-\Delta\dlq-\Delta\dvq)]\|_{L^1}\\
    \leq& \lambda^{-40\alpha}_{q+1}\delta_{q+2},
    \end{align*}
    and
     \begin{align*}
    \|\RR_{q+1}\|_{L^1}\leq & \|\Rosc+\mathcal{R}[(\del_t\wtq+\del_t\wvq)+(\del_t\wlq-\Delta\wlq-\Delta\wvq) ]\|_{L^1}\notag\\ &+ \|\Rlinear- \mathcal{R}[(\del_t\wtq+\del_t\wvq)+(\del_t\wlq-\Delta\wlq-\Delta\wvq)]\|_{L^1}\notag\\
   \leq & \lambda^{-40\alpha}_{q+1}\delta_{q+2}.
\end{align*}
This fact shows that \eqref{e:RR_q-C0} holds by replacing $q$ with $q+1$.

\section{Energy iteration}
\label{s:energy-iteration}

In this section, we show that \eqref{e:energy-q-estimate} and \eqref{e:energy-h-estimate} hold  with $q$ replaced by $q+1$.
\begin{prop}[Energy estimate]
\label{p:energy}For all $t\in[1-\tau_q,T]$, we have
    \[\bigg|e(t)-\int_{\mathbb T^3} (|v_{q+1}|^2+|b_{q+1}|^2)\dd x -\tfrac{\delta_{q+2}}2\bigg| \le \frac{\delta_{q+2}}{100} .\]
\end{prop}
\begin{proof}
For $t\in[1-\tau_q,T]$, the total energy error can be rewritten as follows:
\begin{align}\label{energy0}
 &e(t)-\int_{\mathbb T^3} (|v_{q+1}|^2+|b_{q+1}|^2)\dd x-\frac{\delta_{q+2}}2\notag\\
=& [e(t)-\int_{\mathbb T^3}( |\vv_q|^2+|\bb_q|^2)\dd x -\tfrac{\delta_{q+2}}2] -\int_{\mathbb T^3} (|w_{q+1}|^2+|d_{q+1}|^2)\dd x\notag- 2 \int_{\mathbb T^3} (w_{q+1}\cdot \vv_q +d_{q+1}\cdot \bb_q) \dd x\notag\\
=&[e(t)-\int_{\mathbb T^3}(|\vv_q|^2+|\bb_q|^2 )\dd x -\tfrac{\delta_{q+2}}2]+e_{low,0} \notag\\
&-\int_{\mathbb T^3} (\wpq+\whq)\cdot(\wpq+\whq)+(\dpq+\dhq)\cdot(\dpq+\dhq)\dd x,
\end{align}
where
\begin{align}
e_{low,0}:=&-\int_{\mathbb T^3}\big((\wpq+\whq)\cdot(\wvq+\wlq+\wtq+\wcq+\wsq)\\
&+(\wvq+\wlq+\wtq+\wcq+\wsq)\cdot w_{q+1}\notag\\
&+(\dpq+\dhq)\cdot(\dvq+\dlq+\dtq+\dcq+\dsq)\\
&+(\dvq+\dlq+\dtq+\dcq+\dsq)\cdot d_{q+1}\big)\dd x\notag\\
&-2 \int_{\mathbb T^3} w_{q+1}\cdot \vv_q +d_{q+1}\cdot \bb_q \dd x.
\end{align}
For the last term on the right-hand side of equality \eqref{energy0}, using $\wpqu\cdot\wpqb=\wpq\cdot\whq=\dpq\cdot\dhq=0$ and $\tr\RRR_q=0$, we have
\begin{align}\label{energy2}
    &\int_{\mathbb T^3} (\wpq+\whq)\cdot(\wpq+\whq)+(\dpq+\dhq)\cdot(\dpq+\dhq) \dd x\notag\\
    =& \int_{\mathbb T^3} \tr (\wpqu\otimes \wpqu)+(|\wpqb|^2+|\dpq|^2+|\dhq|^2 +|\whq|^2) \dd x \notag\\
    = &\int_{\mathbb T^3} \tr\Big[\sum_{l } \eta^2_{l}\rho_{v,q} \sum_{k \in \Lambda_v} a^2_{k} \psi^2_k\phi^2_{k,\bar{k},\bar{\bar{k}}}\bar{k}\otimes \bar{k}\Big]  \dd x\notag \\
    &+\int_{\mathbb T^3}\Big(|\wpqb|^2+|\dpq|^2+|\dhq|^2 +|\whq|^2\Big) \dd x \notag \\
    =&
    \int_{\mathbb T^3} \tr\Big[\sum_{l } \eta^2_{l}(\rho_{v,q}{\rm Id}-\RRR_q) \Big]  \dd x+\int_{\mathbb T^3}\Big(|\wpqb|^2+|\dpq|^2+|\dhq|^2 +|\whq|^2\Big) \dd x\notag \\
     &+\int_{\mathbb T^3} \tr\Big[\sum_{l;k \in \Lambda_{v }} \eta^2_{l}\rho_{v,q} a^2_{k}  [\mathbb{P}_{>0}( \psi^2_k)\phi^2_{k,\bar{k},\bar{\bar{k}}}+\mathbb{P}_{>0}\phi^2_{k,\bar{k},\bar{\bar{k}}}]\bar{k}\otimes \bar{k}\Big]\dd x\notag \\
     =&
    3\int_{\mathbb T^3} \sum_{l } \eta^2_{l}\rho_{v,q}   \dd x+\int_{\mathbb T^3}\Big(|\wpqb|^2+|\dpq|^2+|\dhq|^2 +|\whq|^2\Big) \dd x\notag \\
     &+\int_{\mathbb T^3} \tr\Big[\sum_{l;k \in \Lambda_{v }} a^2_{v,l,k}  [\mathbb{P}_{>0}( \psi^2_k)\phi^2_{k,\bar{k},\bar{\bar{k}}}+\mathbb{P}_{>0}\phi^2_{k,\bar{k},\bar{\bar{k}}}]\bar{k}\otimes \bar{k}\Big]\dd x.
\end{align}
Since $\eta_{-1}(t)=\aleph(t)=0$ for $t\in[1-\tau_q,T]$, we rewrite $ \rho_{v,q}$ as
\begin{align}
 \rho_{v,q} (t) &=\chi_{v}\Big( \delta_{q+1}\aleph(t)+\frac{ \rho_q(t)(1-\aleph(t))}{ \eta_{-1}(t) + \sum_{l= 1}^{N_q} \int_{\mathbb T^3}\eta^2_l\chi_{v} ( x,t)\dd x }\Big)\notag\\
 &=\frac{ \chi_{v}\rho_q(t)}{  \sum_{l= 1}^{N_q} \int_{\mathbb T^3}\eta^2_l\chi_{v} (\tilde x,t)\dd \tilde x }.
\end{align}
Recalling the definition of $\rho_q(t)$ in \eqref{e0:energy-gap}, we deduce that
\begin{align}\label{energy3}
&3\int_{\mathbb T^3} \sum_{l } \eta^2_{l}\rho_{v,q}   \dd x
     +\int_{\mathbb T^3}(|\wpqb|^2+|\dpq|^2+|\dhq|^2 +|\whq|^2) \dd x\notag \\
=&3\rho_q(t)+E(t)=e(t) - \int_{\mathbb T^3} |\vv_q|^2+|\bb_q|^2 \dd x - \frac{\delta_{q+2}}2.
\end{align}
\eqref{energy3} helps us  rewrite \eqref{energy2} as
\begin{align}\label{e:energy-wpq+dpq}
&\int_{\mathbb T^3} (\wpq+\whq)\cdot(\wpq+\whq)+(\dpq+\dhq)\cdot(\dpq+\dhq) \dd x\notag\\
=&e(t) - \int_{\mathbb T^3} |\vv_q|^2+|\bb_q|^2 \dd x - \tfrac{\delta_{q+2}}2\notag\\
&+\tr\int_{\mathbb T^3}\sum_{l;k \in \Lambda_{v}} a^2_{v,l,k} [\mathbb{P}_{>0}(\psi^2_k)
     (\phi^2_{k,\bar{k},\bar{\bar{k}}})+\mathbb{P}_{>0}
     (\phi^2_{k,\bar{k},\bar{\bar{k}}})]\bar{k}\otimes \bar{k}\dd x.
\end{align}

Next, using
$$\Big|\int_{\mathbb T^3}f\mathbb{P}_{\geq c}gdx\Big|=\Big|\int_{\mathbb T^3}|\nabla|^Lf|\nabla|^{-L}
\mathbb{P}_{\geq c}gdx\Big|\leq c^{-L}\|g\|_{L^2}\|f\|_{H^L}$$
with $L$ sufficiently large, we obtain
\begin{align}\label{high frequence1}
&\Big|\tr\int_{\mathbb T^3}\sum_{l;k \in \Lambda_{v}} a^2_{v,l,k} [\mathbb{P}_{>0}(\psi^2_k)
     (\phi^2_{k,\bar{k},\bar{\bar{k}}})+\mathbb{P}_{>0}
     (\phi^2_{k,\bar{k},\bar{\bar{k}}})]\bar{k}\otimes \bar{k}\dd x\Big|\notag\\
\lesssim& {\lambda}^{L(\frac{1}{128}+\frac{7}{8})}_{q+1}{\lambda}^{-L}_{q+1}+\ell^{-5L}_{q}{\lambda}^{-\frac{L}{128}}_{q+1}\ll \frac{1}{10000}\delta_{q+2}
\end{align}
and
\begin{align}\label{energy5}
\Big|\int_{\mathbb T^3} w_{q+1}\cdot \vv_q+d_{q+1}\cdot \bb_q \dd x\Big|
\lesssim \lambda^{-5\alpha}_{q+1}\delta_{q+2}\ll\frac{1}{10000}\delta_{q+2}.
 \end{align}

From Remark \ref{xiao}, we deduce that
 \begin{align}\label{energy6}
&\Big|\int_{\mathbb T^3} (\wpq+\whq)\cdot(\wlq+\wtq+\wcq+\wsq)+(\wlq+\wtq+\wcq+\wsq)\cdot w_{q+1}\notag\\
&+(\dpq+\dhq)\cdot(\dlq+\dtq+\dcq+\dsq)+(\dlq+\dtq+\dcq+\dsq)\cdot d_{q+1}\dd x\Big|\notag\\
\lesssim &\lambda^{-50\alpha}_{q+1}\delta_{q+2}\ll\frac{1}{10000}\delta_{q+2}.
 \end{align}
Combining  the above inequality with \eqref{energy5}, we show that
\begin{align}\label{small terms-elow}
|e_{low}|\leq\frac{1}{2000}\delta_{q+2}.
 \end{align}
Finally,  putting \eqref{e:energy-wpq+dpq},~\eqref{high frequence1} and
  \eqref{small terms-elow} into \eqref{energy0}, we have \begin{align}\label{energy4}
  &\bigg|e(t)-\tfrac{\delta_{q+2}}2-\int_{\mathbb T^3} (|v_{q+1}|^2+|b_{q+1}|^2)\dd x \bigg|\notag\\
  =&\bigg|\Big[e(t)-\int_{\mathbb T^3}(|\vv_q|^2+|\bb_q|^2) \dd x -\tfrac{\delta_{q+2}}2\Big] \notag\\
&-\int_{\mathbb T^3} (\wpq+\whq)\cdot(\wpq+\whq)+(\dpq+\dhq)\cdot(\dpq+\dhq)\dd x+e_{low}\bigg|\notag\\
=&\bigg|\Big[e(t)-\int_{\mathbb T^3}(|\vv_q|^2+|\bb_q|^2 )\dd x -\tfrac{\delta_{q+2}}2\Big] \notag\\
&-\Big[e(t) - \int_{\mathbb T^3} (|\vv_q|^2+|\bb_q|^2) \dd x - \frac{\delta_{q+2}}2\Big]-\int_{\mathbb T^3} \tr\Big(\sum_{l;k \in \Lambda_{v} } a^2_{v,l,k} \mathbb{P}_{>0}(\phi^2_{k,\bar{k},\bar{\bar{k}}})\bar{k}\otimes \bar{k}\Big)\dd x +e_{low}\bigg|\notag\\
\leq &\Big|\int_{\mathbb T^3} \tr\Big(\sum_{l;k \in \Lambda_{v} } a^2_{v,l,k} \mathbb{P}_{>0}(\phi^2_{k,\bar{k},\bar{\bar{k}}})\bar{k}\otimes \bar{k}\Big)\dd x\Big|+\Big|e_{low}\Big|\notag\\
\leq&\frac{1}{1000}\delta_{q+2},~~\forall t\in[1-\tau_q,T].
   \end{align}
This completes the proof of Proposition \ref{p:energy}.
\end{proof}

\begin{prop}[Helicity estimate]
\label{h:energy}For all $t\in[1-\tau_q,T]$, we have
    \[\bigg|h(t)-\int_{\mathbb T^3} v_{q+1}\cdot b_{q+1}\dd x -\frac{\delta_{q+2}}{200}\bigg| \le \frac{\delta_{q+2}}{1000} .\]
\end{prop}
\begin{proof}
Similar to the proof of Proposition \ref{p:energy}, it suffices to control the helicity for $t\in[1-\tau_q,T]$. Due to
\begin{align}
 h_{b,q} (t) =\frac{ h_q(t)}{  \sum_{l=0}^{N_q} \int_{\mathbb T^3}\eta^2_l  ( x,t)\dd x },~~~t\in[1-\tau_q,T],
\end{align}
we rewrite that
\begin{align}\label{hel0}
 &h(t)-\int_{\mathbb T^3} v_{q+1}\cdot b_{q+1}\dd x-\tfrac{\delta_{q+2}}{200}\notag\\
=& \Big[h(t)-\int_{\mathbb T^3} \vv_q\cdot \bb_q\dd x -\tfrac{\delta_{q+2}}{200}\Big]
-\int_{\mathbb T^3} w_{q+1}\cdot d_{q+1} \dd x
-\int_{\mathbb T^3} w_{q+1}\cdot \bb_q +\vv_q\cdot d_{q+1} \dd x\notag\\
:=&\Big[h(t)-\int_{\mathbb T^3}\vv_q\cdot \bb_q \dd x -\tfrac{\delta_{q+2}}{200}\Big]  -\int_{\mathbb T^3} \whq\cdot \dhq \dd x-\int_{\mathbb T^3}\wpq\cdot\dpq\dd x+h_{low,0}\notag\\
=&\Big[h(t)-\int_{\mathbb T^3}\vv_q\cdot \bb_q \dd x -\tfrac{\delta_{q+2}}{200}\Big]  -\int_{\mathbb T^3} \whq\cdot \dhq \dd x-0+h_{low,0},
\end{align}
where
\begin{align}
h_{low,0}:=&-\int_{\mathbb T^3}
(w_{q+1}-\whq-\wpq)\cdot (d_{q+1}-\dhq-\dpq)+(\whq+\wpq)\cdot (d_{q+1}-\dhq-\dpq)\notag\\
&+(w_{q+1}-\whq-\wpq)\cdot (\dhq+\dpq)\dd x.
\end{align}
By the definitions of $\whq$ and $\dhq$ in \eqref{e:defn-whq}, we obtain that
\begin{align}\label{hel2}
&\int_{\mathbb T^3}\whq\cdot\dhq \dd x= \int_{\mathbb T^3}\sum_{l;k\in\Lambda_s} \eta^2_{l}h_{b,q}(t)   \psi^2_k\phi^2_{k,\bar{k},\bar{\bar{k}}}\bar{k}\dd x \notag\\
    =&h_{b,q}\int_{\mathbb T^3}\sum_{l;k\in\Lambda_s}\eta^2_{l}\dd x+
    \Big\{h_{b,q}\int_{\mathbb T^3}\sum_{l;k\in\Lambda_s}\eta^2_{l}   \mathbb{P}_{>0}(\psi^2_k)\phi^2_{k,\bar{k},\bar{\bar{k}}}\bar{k} \dd x
    +h_{b,q}\int_{\mathbb T^3}\sum_{l;k\in\Lambda_s}\eta^2_{l}   \mathbb{P}_{>0}(\phi^2_{k,\bar{k},\bar{\bar{k}}})\bar{k}\dd x\Big\}\notag\\
    :=&h_q(t)+h_{low,1}.
\end{align}
Using the standard integration by parts and Remark \ref{xiao}, one deduces that
\begin{align}\label{hel3}
|h_{low,0}|+|h_{low,1}|\lesssim \lambda^{-50\alpha}_{q+1}\delta_{q+2}.
\end{align}
Since $h_q(t)=h(t)-\int_{\mathbb T^3}\vv_q\cdot\bb_q \dd x -\frac{\delta_{q+2}}{200}$, plugging \eqref{hel2} and \eqref{hel3} into \eqref{hel0} yields that
$$\bigg|h(t)-\int_{\mathbb T^3} v_{q+1}\cdot b_{q+1}\dd x -\frac{\delta_{q+2}}{200}\bigg|\le |h_{low,0}|+|h_{low,1}|\lesssim \lambda^{-50\alpha}_{q+1}\delta_{q+2} \le \frac{\delta_{q+2}}{1000}.$$
Therefore, we complete the proof of the Proposition \ref{h:energy}.
\end{proof}

In conclusion, collecting all the results in Section 3--Section 6, we prove Proposition \ref{p:main-prop}.

\noindent\textbf{Acknowledgements}
 We thank the anonymous referee  and the associated editor for their invaluable comments
which helped to improve the paper. This work was supported by the National Key Research and Development Program of China (No. 2020YFA0712900) and NSFC Grant 11831004.

\appendix
\section{Appendix}
\subsection{Inverse divergence operator}
\label{ss:inverse-div}
Recalling the following pseudodifferential operator of order $-1$ as in  \cite{2Beekie}:
\begin{align}
    {\mathcal{R} v}^{kl} = \partial_k\Delta^{-1}v^l+ \partial_l\Delta^{-1}v^k-\frac{1}{2}
    (\delta_{kl}+\partial_k\partial_l\Delta^{-1})\div\Delta^{-1}v,\label{Ru}\\
      {\mathcal{R}_{a} u}^{ij} = \epsilon_{ijk}(-\Delta)^{-1}(\curl u)_k,\label{RRu}
\end{align}
where $\epsilon_{ijk}$ is the Levi-Civita tensor, $i,j,k\in\{1,2,3\}$.  One can easily verify that $\mathcal{R}$ is a matrix-valued right inverse of the divergence operator, such that $\div\mathcal{R}v=v$ for mean-free vector fields $v$. In addition, $\mathcal Rv$ is traceless and symmetric. While $\mathcal{R}_{a}$ is also a matrix-valued right inverse of the divergence operator, such that $\div\mathcal{R}_au=u$ for divergence free vector fields $u$. And $\mathcal{R}_{a}u$ is anti-symmetric.

\subsection{Some notations in geometric lemmas}\label{Geo}
Using the same idea as in \cite{2Beekie}, we give the definitions of $\Lambda_b$, $\Lambda_v$ and  $\Lambda_s$  by
\begin{align*}
&\Lambda_b=\Big\{e_1,e_2,e_3,-\tfrac{4}{5}e_2-\tfrac{3}{5}e_3,
\tfrac{3}{5}e_1+\tfrac{4}{5}e_2\Big\},\\
&\Lambda_v=\{\tfrac{12}{13}e_1\pm\tfrac{5}{13}e_2,\tfrac{5}{13}e_1\pm\tfrac{12}{13}e_3,
\tfrac{12}{13}e_2\pm\tfrac{5}{13}e_3\},\\
&\Lambda_s=\Big\{\tfrac{9}{41}e_1+\tfrac{40}{41}e_2\Big\}.
\end{align*}
 For fixed $k\in\Lambda_b\cup \Lambda_v\cup \Lambda_s$, one can obtain a unique triple $(k,\bar{k},\bar{\bar{k}})$. See the following table:
 \begin{table}[ht]
\begin{tabular}{p{1.7cm}|p{1.7cm}|p{0.4cm}|p{1.7cm}|p{1.7cm}|p{0.4cm}|p{1.7cm}|p{1.7cm}|p{0.4cm}}
\toprule[1pt]
\multicolumn{3}{c|} {$k\in \Lambda_b, \,\,\,k\perp \bar{k}\perp \bar{\bar{k}}$} &\multicolumn{3}{|c|} {$k\in \Lambda_v, \,\,\,k\perp \bar{k}\perp \bar{\bar{k}}$}&\multicolumn{3}{|c}{$k\in \Lambda_s, \,\,\,k\perp \bar{k}\perp \bar{\bar{k}}$}\\\hline
$k$&$\bar{k}$&$\bar{\bar{k}}$&$k$&$\bar{k}$&$\bar{\bar{k}}$ &$k$&$\bar{k}$&$\bar{\bar{k}}$\\\hline
$e_1$&{$ e_2$}& $e_3$&$\tfrac{12}{13}e_1\pm\tfrac{5}{13}e_2$&{$ \tfrac{5}{13}e_1\mp\tfrac{12}{13}e_3$}& $e_3$&$\tfrac{9}{41}e_1+\tfrac{40}{41}e_2$& $\tfrac{40}{41}e_1-\tfrac{9}{41}e_2$& $e_3$\\
$e_2$&{$ e_3$}&$e_1$& $\tfrac{5}{13}e_1\pm\tfrac{12}{13}e_3$&$\tfrac{5}{13}e_1\mp\tfrac{12}{13}e_3$&{$ e_2$}&&&\\
$e_3$&{$ e_1$}& $e_2$&$\tfrac{12}{13}e_2\pm\tfrac{5}{13}e_3$&{$ \tfrac{12}{13}e_2\mp\tfrac{5}{13}e_3$}& $e_1$&&&\\
$-\tfrac{4}{5}e_2-\tfrac{3}{5}e_3$&{$-\tfrac{3}{5}e_2-\tfrac{4}{5}e_3$}&$ e_1 $&  & & &&&\\
$\tfrac{3}{5}e_1+\tfrac{4}{5}e_2$&{$\tfrac{4}{5}e_1-\tfrac{3}{5}e_2$}& $e_3$& & & &&&\\\bottomrule[1pt]
\end{tabular}
\end{table}


\newpage
\subsection{Two useful lemmas}
\begin{lem}[\cite{2Beekie,13Nonuniqueness}]\label{holder1}
Fix positive integers $L,\kappa,\lambda\geq1$ such that
$$\frac{\lambda}{\kappa}\leq\frac{1}{30}~~~and
~~~\frac{\lambda^{L+4}}{\kappa^L}\leq\frac{1}{30}.$$
Let $p\in\{1,2\}$, and $f$ be a $\mathbb{T}^3$-periodic function such that there exists a constant $C_f$ such that
$$\|D^jf\|_{L^p}\leq C_f\lambda^j,~~~ j\in[1,L+4].$$
Moreover, assume that $g$ is a $(\mathbb{T}/\kappa)^3$-periodic function. Then, we have
$$\|fg\|_{L^p}\lesssim C_f\|g\|_{L^p},$$
where the implicit constant is universal.
\end{lem}

One can obtain the following Lemma after some modifications in Lemma B.1 in \cite{13Nonuniqueness}.
\begin{lem}\label{holder2}
Fixed $\kappa>\lambda\geq1$ and $p\in (1,2]$. Assume that there exists an integer $L$ with $\kappa^{L-2}>\lambda^{L}$. Let $f$ be a $\mathbb{T}^3$-periodic function so that there exists a constant $C_f$ such that
$$\|D^jf\|_{L^{p_1}}\lesssim \lambda^j\|f\|_{L^{p_1}},~~~ j\in[0,L].$$
Assume further that $\int_{\mathbb{T}^3}f(x)\mathbb{P}_{\geq \kappa}g(x)dx=0$ and $g$ is a $(\mathbb{T}/\kappa)^3$-periodic function. Then, we have
$$\||\nabla|^{-1}(f\mathbb{P}_{\geq \kappa}g)\|_{L^p}\lesssim  \tfrac{\|f\|_{L^{p_1}}\|g\|_{L^{p_2}}}{\kappa},~~~~\tfrac{1}{p}=\tfrac{1}{p_1}+\tfrac{1}{p_2},$$
where the implicit constant is universal.
\end{lem}

\addcontentsline{toc}{section}{\refname}
\bibliographystyle{plain}
\bibliography{reference}

\begin{thebibliography}{10}

\bibitem{leray}
D.~Albritton, E.~Bru\'{e}, and M.~Colombo.
\newblock Non-uniqueness of leray solutions of the forced {N}avier-{S}tokes
  equations.
\newblock {\em Ann.of Math.}, 196(1):415--455, 2022.

\bibitem{1Hydrodynamic}
H.~Aluie.
\newblock {\em Hydrodynamic and magnetohydrodynamic turbulence: Invariants,
  cascades and locality}.
\newblock {P} h.D. thesis, {J}ohns {H}opkins {U}niversity, 2009.

\bibitem{book}
H.~Bahouri, J.~Y. Chemin, and R.~Danchin.
\newblock {\em Fourier analysis and nonlinear partial differential equations},
  volume 343 of {\em Grundlehren der Mathematischen Wissenschaften [Fundamental
  Principles of Mathematical Sciences]}.
\newblock Springer, Heidelberg, 2011.

\bibitem{2Beekie}
R.~Beekie, T.~Buckmaster, and V.~Vicol.
\newblock Weak solutions of ideal {MHD} which do not conserve magnetic
  helicity.
\newblock {\em Ann. of PDE.}, 6(1):1--40, 2020.

\bibitem{2018Wild}
T.~Buckmaster, M.~Colombo, and V.~Vicol.
\newblock {Wild solutions of the Navier-Stokes equations whose singular sets in
  time have Hausdorff dimension strictly less than 1}.
\newblock {\em J. Eur. Math. Soc.}, 24(9):3333--3378, 2021.

\bibitem{zbMATH07038033}
T.~{Buckmaster}, C.~{D}. Lellis, L.~{Sz\'ekelyhidi}, and V.~{Vicol}.
\newblock {Onsager's conjecture for admissible weak solutions}.
\newblock {\em {Commun. Pure Appl. Math.}}, 72(2):229--274, 2019.

\bibitem{MV}
T.~Buckmaster, N.~Masmoudi, M.~Novack, and V.~Vicol.
\newblock Non-conservative ${H}^{1/2-}$ weak solutions of the incompressible
  {3D} euler equations..
\newblock {\em arXiv:2101.09278, 2021.}, 2022.

\bibitem{13Nonuniqueness}
T.~Buckmaster and V.~Vicol.
\newblock {Nonuniqueness of weak solutions to the {Navier}-{Stokes} equation}.
\newblock {\em {Ann. of Math.}}, 189(2):101--144, 2019.

\bibitem{34Caflisch}
R.~E. Caflisch, I.~Klapper, and G.~Steele.
\newblock {Remarks on singularities, dimension and energy dissipation for ideal
  hydrodynamics and MHD}.
\newblock {\em {Comm. Math. Phys.}}, 184(2):443--455, 1997.

\bibitem{1Cheskidov}
A.~Cheskidov and X.~Luo.
\newblock {Sharp nonuniqueness for the Navier-Stokes equations}.
\newblock {\em {Invent. math.}}, 229(3):987--1054, 2022.

\bibitem{cwe}
P.~Constantin, Weinan E, and E.~S. Titi.
\newblock Onsager's conjecture on the energy conservation for solutions of
  {E}uler's equation.
\newblock {\em Comm. Math. Phys.}, 165(1):207--209, 1994.

\bibitem{zbMATH06312794}
S.~{Daneri}.
\newblock {Cauchy problem for dissipative H\"older solutions to the
  incompressible Euler equations}.
\newblock {\em {Comm. Math. Phys.}}, 329(2):745--786, 2014.

\bibitem{zbMATH07370998}
S.~{Daneri}, E.~{Runa}, and L.~{S}zs\'{e}kelyhidi {Jr}.
\newblock {Non-uniqueness for the Euler equations up to Onsager's critical
  exponent}.
\newblock {\em {Ann. of PDE}}, 7(1):44, 2021.
\newblock No 8.

\bibitem{zbMATH06710292}
S.~{Daneri} and L.~Sz\'ekelyhidi~{Jr}.
\newblock {Non-uniqueness and \(h\)-principle for H\"older-continuous weak
  solutions of the Euler equations}.
\newblock {\em {Arch. Ration. Mech. Anal.}}, 224(2):471--514, 2017.

\bibitem{Rosa2021DimensionOT}
L.~De~Rosa and S.~Haffter.
\newblock Dimension of the singular set of wild h\"older solutions of the
  incompressible euler equations.
\newblock {\em {Nonlinearity}}, 35(10):5150--5192, 2021.

\bibitem{35Faraco}
D.~Faraco and S.~Lindberg.
\newblock Proof of {T}aylor's conjecture on magnetic helicity conservation.
\newblock {\em {Comm. Math. Phys.}}, 273(2):707--738, 2021.

\bibitem{36Faraco}
D.~Faraco, S.~Lindberg, and L.~Szs\'{e}kelyhidi~Jr.
\newblock {Bounded solutions of ideal {MHD} with compact support in
  space-time}.
\newblock {\em Arch. Ration. Mech. Anal.}, 239(1):51--93, 2021.

\bibitem{37Faraco}
D.~Faraco, S.~Lindberg, and L.~Szs\'{e}kelyhidi~Jr.
\newblock Magnetic helicity, weak solutions and relaxation of ideal {MHD}.
\newblock {\em arXiv:2109.09106}, 2021.

\bibitem{Isett}
P.~Isett.
\newblock {A proof of Onsager's conjecture}.
\newblock {\em {Ann. of Math.}}, 188(3):871--963, 2018.

\bibitem{45Kang}
E.~Kang and J.~Lee.
\newblock Remarks on the magnetic helicity and energy conservation for ideal
  magneto-hydrodynamics.
\newblock {\em Nonlinearity.}, 20(11):2681--2689, 2007.

\bibitem{KMY}
C.~Khor, C.~Miao, and W.~Ye.
\newblock Infinitely many non-conservative solutions for the three-dimensional
  euler equations with arbitrary initial data in ${C}^{1/3-}$.
\newblock {\em arXiv.2204.03344.}, 2022.

\bibitem{27DeLellis}
C.~De Lellis and Jr. L.~Szs\'{e}kelyhidi.
\newblock {The Euler equations as a differential inclusion}.
\newblock {\em {Ann. of Math.}}, 170(3):1417--1436, 2009.

\bibitem{28DeLellis}
C.~De Lellis and L.~{S}zs\'{e}kelyhidi {Jr}.
\newblock {On admissibility criteria for weak solutions of the Euler
  equations}.
\newblock {\em {Arch. Ration. Mech. Anal.}}, 195(1):225--260, 2010.

\bibitem{lyc}
Y.~Li, Z.~Zeng, and D.~Zhang.
\newblock Non-uniqueness of weak solutions to {3D} generalized
  magnetohydrodynamic equations.
\newblock {\em J. Math. Pures. Appl.}, 165:232--285, 2022.

\bibitem{luotianwen}
T.~Luo and E.S. Titi.
\newblock Non-uniqueness of weak solutions to hyperviscous {N}avier-{S}tokes
  equations - on sharpness of j.-{L}. lions exponent.
\newblock {\em Calc. Var. Partial Differential Equations.}, 59(3):1--15, 2020.

\bibitem{NV}
M.~Novack and V.~Vicol.
\newblock An intermittent onsager theorem.
\newblock {\em arXiv.2204.03344.}, 2022.

\bibitem{65Taylor}
J.~B. Taylor.
\newblock {Relaxation and magnetic reconnection in plasmas}.
\newblock {\em {Reviews of Modern Physics.}}, 58(3):741, 1986.

\bibitem{67Wu}
J.~Wu.
\newblock {Generalized MHD equations}.
\newblock {\em {J. Diff. Equ.}}, 195(2):284--312, 2003.

\end{thebibliography}

\end{document}